\documentclass{birkjour}

\usepackage{cite}
\usepackage{amssymb}
\usepackage{amsmath}
\usepackage{amsfonts}
\usepackage{amsthm}
\usepackage{enumerate}
\usepackage{colonequals}
\usepackage{mathrsfs} 
\usepackage{color}
\usepackage{enumitem}
\usepackage{mathtools}
\usepackage[all,cmtip]{xy}
\usepackage{tikz-cd}

\usepackage{url}

\usepackage[mathscr]{euscript}

\newtheorem{theorem}{Theorem}[section]

\newtheorem{proposition}{Proposition}[section]

\newtheorem*{theorem*}{Theorem}

\newtheorem{definition}{Definition}[section]

\newtheorem{lemma}{Lemma}[section]

\numberwithin{equation}{section}

\allowdisplaybreaks[2]

\newcommand{\dee}{\mathrm{d}}
\newcommand{\deee}{\hspace{2 pt} \mathrm{d}}
\newcommand{\rest}{\upharpoonright}

\newcommand{\ov}[1]{\overline{#1}}

\newcommand{\nml}{\left \vert \left \vert}
\newcommand{\nmr}{\right \vert \right \vert}

\renewcommand{\hat}{\widehat}
\renewcommand{\tilde}{\widetilde}

\allowdisplaybreaks

\begin{document}

\title{The extension problem in free harmonic analysis}
\author{Peter Burton and Kate Juschenko}

\maketitle

\begin{abstract}

This paper studies certain aspects of harmonic analysis on nonabelian free groups. We focus on the concept of a positive definite function on the free group and our primary goal is to understand how such functions can be extended from balls of finite radius to the entire group. Previous work showed that such extensions always exist and we study the problem of simultaneous extension of multiple positive definite functions. More specifically, we define a concept of `relative energy' which measures the proximity between a pair of positive definite functions, and show that a pair of positive definite functions on a finite ball can be extended to the entire group without increasing their relative energy. The proof is analytic, involving differentiation of noncommutative Szeg\"{o} parameters.

\end{abstract}

\tableofcontents

\section{Introduction} \label{part.intro}

Noncommutative harmonic analysis presents numerous challenges that are absent in the commutative setting due to the existence of higher dimensional irreducible unitary representations. These challenges are greatest in the context of nonamenable discrete groups, where the unitary dual is extremely complicated. One situation where  harmonic analysis on nonamenable discrete groups has proved somewhat tractable is that of free groups. Previous work in this direction includes \cite{MR2316876}, \cite{MR590864}, \cite{MR573099}, \cite{MR1152801}, \cite{MR665019}, \cite{MR710827} and  \cite{MR520930}.\\
\\
A basic task in harmonic analysis is to extend positive definite functions from finite subsets of a group to the entire group. Such an extension always exists for positive definite functions defined on intervals in $\mathbb{Z}$ via the construction with Szeg\"{o} parameters, but it already can fail to exist for positive definite functions defined on disks in $\mathbb{Z}^2$. It seems freeness is essential to the existence of such extensions and it was established in \cite{MR2316876} that a positive definite function defined on a ball of finite radius in a finitely generated free group always has an extension to the entire group. \\
\\
In this paper we study the question of simultaneous extension of positive definite functions on free groups. We write $\mathbb{F}_r$ for the free group of rank $r$. There is a concept of the distance between two positive definite functions that we think of as the relative energy required to transform between the configurations of vectors represented by the functions. There is then a natural question about the existence of simultaneous extensions without energy increase which arises in the study of intertwining maps between unitary representations of $\mathbb{F}_r$, for example when considering representations of $\mathbb{F}_r \times \mathbb{F}_r$ in the context of Connes' embedding conjecture. Our main result is that a pair of positive definite functions can be extended simultaneously without increasing their relative energy. The construction relies on an variational argument about certain numerical `noncommutative Szeg\"{o} parameters' that govern the recursive extension procedure.\\
\\
In this paper we never consider free groups of rank higher than two, and consequently we denote the free group of rank two simply by $\mathbb{F}$. It is easy to see that the results can be generalized to finitely generated free groups of ranks higher than two. We now present the definitions needed to formalize the previous discussion and to state our main results.

\subsection{Preliminary information} \label{subsec.basicinfo}

Fix a pair of free generators $a$ and $b$ for $\mathbb{F}$ and endow $\mathbb{F}$ with the standard Cayley graph structure corresponding to left multiplication by these generators. If $\Gamma$ is a quotient of $\mathbb{F}$ we identify $a$ and $b$ with their images in $\Gamma$. We write $e$ for the identity of $\mathbb{F}$. \\
\\
We consider the word length associated to $a$ and $b$, which we denote by $|\cdot|$. For $r \in \mathbb{N}$ let $\mathbb{B}_r = \{g \in \mathbb{F}:|g| \leq r\}$ be the ball of radius $r$ around $e$. Write $K_r$ for the cardinality of $\mathbb{B}_r$.

\begin{definition} Let $d \in \mathbb{N}$ and let $F$ be a finite subset of $\mathbb{F}$. We define a function $\mathsf{C}:F \to \mathbb{C}$ to be \textbf{positive definite} if we have the fundamental inequality \begin{equation} \label{eq.posdef} \sum_{g,h \in E} \alpha(g) \ov{\alpha(h)} \mathsf{C}(h^{-1}g) \geq 0 \end{equation} for every subset $E$ of $\mathbb{F}$ with $E^{-1}E \subseteq F$ and every function $\alpha:E \to \mathbb{C}$.\\
\\
We define $\mathsf{C}$ to be \textbf{strictly positive definite} if $\mathsf{C}$ is positive definite and the inequality in (\ref{eq.posdef}) is saturated only when $\alpha$ is identically $0$. We define a function $\mathsf{C}:\mathbb{F} \to \mathbb{C}$ to be (strictly) positive definite if $\mathsf{C} \rest F$ is (strictly) positive definite for every finite $F \subseteq \mathbb{F}$. \end{definition}

A positive definite function on the free group can be thought of as a noncommutative analog of an infinite positive definite Toeplitz matrix. We will always assume the following normalization condition.

\begin{definition} If $\mathsf{C}$ is a positive definite function whose domain contains the identity $e$ of $\mathbb{F}$, we define $\mathsf{C}$ to be \textbf{normalized} if $\mathsf{C}(e) = 1$. \end{definition}

We denote the space of normalized strictly positive definite functions $\mathsf{C}:\mathbb{B}_r \to \mathbb{C}$ by $\mathrm{NSPD}_r$. Note that $\mathbb{B}_r^{-1}\mathbb{B}_r = \mathbb{B}_{2r}$. Therefore if $\mathsf{C} \in \mathrm{NSPD}_{2r}$ we can regard it as a positive definite kernel on the set $\mathbb{B}_r$. By Theorem C.2.3 in \cite{MR2415834} there exists a Hilbert space $\mathscr{X}(\mathsf{C})$ and a function $\Phi_\mathsf{C}: \mathbb{B}_r \to \mathscr{X}(\mathsf{C})$ such that \begin{equation} \label{eq.mari-0} \langle \Phi_\mathsf{C}(g),\Phi_\mathsf{C}(h) \rangle = \mathsf{C}(h^{-1}g) \end{equation} for all $g,h \in \mathbb{B}_r$. Moreover, we may and will assume the range of $\Phi_\mathsf{C}$ spans $\mathscr{X}(\mathsf{C})$. The hypothesis that $\mathsf{C}$ is strictly positive definite ensures that the range of $\Phi_\mathsf{C}$ will be linearly independent and the normalization $\mathsf{C}(e) = 1$ ensures that each vector $\Phi_\mathsf{C}(g)$ has length $1$. We will refer to the range of $\Phi_\mathsf{C}$ as the canonical basis for $\mathscr{X}(\mathsf{C})$.

\begin{definition} We say that $(\mathscr{X}(\mathsf{C}),\Phi_\mathsf{C})$ as above is a \textbf{realization} of $\mathsf{C}$. \end{definition}

We can construct a realization of a positive definite function $\mathsf{C}:\mathbb{F} \to \mathbb{C}$ in the same way, obtaining a Hilbert space $\mathscr{X}(\mathsf{C})$ and a function $\Phi_\mathsf{C}:\mathbb{F} \to \mathscr{X}(\mathsf{C})$ such that the span of the range of $\Phi_\mathsf{C}$ is dense in $\mathscr{X}(\mathsf{C})$.\\
\\
It is clear that given two realizations of the same positive definite function there exists a natural unitary isomorphism from one realization Hilbert space to the other. This isomorphism transforms a canonical basis vector in one realization to the canonical basis vector in another realization having the same index. If $\mathsf{C}:\mathbb{F} \to \mathbb{C}$ is positive definite then the function $g \mapsto \Phi_\mathsf{C}(hg)$ is a realization of $\Phi_\mathsf{C}$ for any $h \in \mathbb{F}$. Thus we may make the following definition.

\begin{definition} Let $d \in \mathbb{N}$ and let $\mathsf{C}:\mathbb{F} \to \mathbb{C}$ be positive definite. Then any realization of $\mathsf{C}$ defines an \textbf{associated unitary representation} of $\mathbb{F}$ on $\mathscr{X}(\mathsf{C})$ denoted by $\rho_\mathsf{C}$ and given by the translation $\rho_\mathsf{C}(h)  \Phi_\mathsf{C}(g) = \Phi_\mathsf{C}(hg)$ for $g,h \in \mathbb{F}$. \end{definition}

We now make the main definition we will be studying.

\begin{definition} Let $\mathsf{C},\mathsf{D} \in \mathrm{NSPD}_{2r}$. Let $(\mathscr{X}(\mathsf{C}),\Phi_\mathsf{C})$ and $(\mathscr{X}(\mathsf{D}),\Phi_\mathsf{D})$ be realizations of $\mathsf{C}$ and $\mathsf{D}$ respectively. Define the \textbf{transport operator} $t[\mathsf{C},\mathsf{D}]:\mathscr{X}(\mathsf{C}) \to \mathscr{X}(\mathsf{D})$ by setting \[ t[\mathsf{C},\mathsf{D}]  \sum_{g \in \mathbb{B}_r} \alpha(g)  \Phi_\mathsf{C}(g)  = \sum_{g \in \mathbb{B}_r} \alpha(g) \Phi_\mathsf{D}(g) \] for functions $\alpha:\mathbb{B}_r \to \mathbb{C}$. We refer to the square of the operator norm of $t[\mathsf{C},\mathsf{D}]$ as the \textbf{relative energy} of the pair $(\mathsf{C},\mathsf{D})$ and denote it by $\mathfrak{e}(\mathsf{C},\mathsf{D})$.\\
\\
If $\mathsf{C},\mathsf{D}:\mathbb{F} \to \mathbb{C}$ are strictly positive definite we define the relative energy of the pair $(\mathsf{C},\mathsf{D})$ to be \[ \sup_{r \in \mathbb{N}} \mathfrak{e}(\mathsf{C} \rest \mathbb{B}_r,\mathsf{D} \rest \mathbb{B}_r) \] We continue to denote it by $\mathfrak{e}(\mathsf{C},\mathsf{D})$. In general we may have $\mathfrak{e}(\mathsf{C},\mathsf{D}) = \infty$. If $\mathfrak{e}(\mathsf{C},\mathsf{D}) < \infty$ then there is a naturally defined transport operator from $\mathscr{X}(\mathsf{C})$ to $\mathscr{X}(\mathsf{D})$, which we continue to denote by $t[\mathsf{C},\mathsf{D}]$. \label{def.transportop} \end{definition}

Note that we have assumed $\mathsf{C}(e) = \mathsf{D}(e) = 1$ so that $\mathfrak{e}(\mathsf{C},\mathsf{D}) \geq 1$ for all $\mathsf{C},\mathsf{D}$  and $\mathfrak{e}(\mathsf{C},\mathsf{D}) = 1$ if and only if $\mathsf{C} = \mathsf{D}$. \\\
\\
The connection between relative energy and intertwining maps between unitary representations of $\mathbb{F}$ comes from the fact that if $\mathsf{C},\mathsf{D}:\mathbb{F} \to \mathbb{C}$ are positive definite and $\mathfrak{e}(\mathsf{C},\mathsf{D}) = 1$ then $t[\mathsf{C},\mathsf{D}]$ is a unitary map intertwining $\rho_\mathsf{C}$ and $\rho_\mathsf{D}$. If $1< \mathfrak{e}(\mathsf{C},\mathsf{D}) < \infty$ then $t[\mathsf{C},\mathsf{D}]$ is still a well-defined intertwining map between the associated representations $\rho_\mathsf{C}$ and $\rho_\mathsf{D}$ although it is no longer unitary.

\subsection{Main result} \label{sec.ilhanomar}

Our main result is the following.

\begin{theorem} \label{thm} Let $r \in \mathbb{N}$ and let $\mathsf{C},\mathsf{D} \in \mathrm{NSPD}_{2r}$. Then there exist positive definite functions $\hat{\mathsf{C}},\hat{\mathsf{D}}: \mathbb{F} \to \mathbb{C}$ with $\hat{\mathsf{C}} \rest \mathbb{B}_{2r} = \mathsf{C}$ and $\hat{\mathsf{D}} \rest \mathbb{B}_{2r} = \mathsf{D}$ and $\mathfrak{e}(\mathsf{C},\mathsf{D}) = \mathfrak{e}(\hat{\mathsf{C}},\hat{\mathsf{D}})$. \end{theorem}

The existence of finitely supported positive definite functions on $\mathbb{Z}$ converging to the constant function $1$ from below gives a simple proof that any pair of positive definite functions defined on the same interval $\mathbb{Z}$ can be extended to all of $\mathbb{Z}$ with a finite increase in relative energy. We do not know whether there is a proof other than ours for the direct analog of Theorem \ref{thm} in the case of $\mathbb{Z}$ (i.e. with no increase in relative energy). We also do not know of a proof other than ours for the statement that two positive definite functions defined on the same finite balls in $\mathbb{F}$ can be extended to all of $\mathbb{F}$ with a finite increase in relative energy.

\subsection{Notation}

We define an ordering $\preceq$ on the sphere of radius $1$ in $\mathbb{F}$ by setting $a \preceq b \preceq a^{-1} \preceq b^{-1}$. From this we obtain a corresponding shortlex linear ordering on all of $\mathbb{F}$, which we continue to denote by $\preceq$. For $g \in \mathbb{F}$ define $\mathcal{I}_g = \bigcup \{\{h,h^{-1}\}: h \preceq g\}$. Define a generalized Cayley graph $\mathrm{Cay}(\mathbb{F},g)$ with vertex set equal to $\mathbb{F}$ by placing an edge between distinct elements $h$ and $\ell$ if and only if $\ell^{-1}h \in \mathcal{I}_g$. Write $g_\uparrow$ for the immediate predecessor of $g$ in $\preceq$ and $g_\downarrow$ for the immediate successor of $g$ in $\preceq$. \\
\\
We write $\mathbb{D}$ for the open unit disk in the complex plane. If $n \in \mathbb{N}$ we write $[n]$ for $\{1,\ldots,n\}$.

\subsection{Acknowledgements}

We thank Lewis Bowen and Rostyslav Kravchenko along with the anonymous referee for helpful comments. 

\section{Review of certain aspects of classical theory} \label{sec.classical}

In Section \ref{sec.classical} we recall some information about classical harmonic analysis which will be relevant to our later arguments. Classical harmonic analysis typically begins with Fourier analysis and the theory of Fourier series on the unit circle $\mathbb{T}$ can be thought of as harmonic analysis on the additive group of integers $\mathbb{Z}$, which from our perspective is the free group of rank one.

\subsection{The fundamental inequality on the integers}

We define a marked unitary representation of a countable discrete group $G$ to be a unitary representation $\rho$ of $G$ on a Hilbert space $\mathscr{X}$ together with a distinguished vector $x \in \mathscr{X}$ which is cyclic in the sense that the span of the set $\{\rho(g)x:g \in G\}$ is dense in $\mathscr{X}$. Any unitary representation of a countable discrete group $G$ can be decomposed into a countable orthogonal sum of subrepresentations, each of which admits a cyclic vector.\\
\\
The family of marked unitary representations of $\mathbb{Z}$ is in canonical one to one correspondence with the the family of finite Borel measures on $\mathbb{T}$. This correspondence is given by placing both of these families of objects in canonical one to one correspondence with a third family of objects. This third family consists of so-called positive definite functions on $\mathbb{Z}$. A positive definite function $C$ on $\mathbb{Z}$ is a function from $\mathbb{Z}$ to the complex numbers which satisfies the fundamental inequality \begin{equation} \sum_{m,n =-N}^N \alpha(m)\ov{\alpha(n)}C(m-n) \geq 0 \label{eq.fundamental-0} \end{equation} for every natural number $N$ and every function $\alpha:\{-N,\ldots N\} \to \mathbb{C}$. One can think of a positive definite function on $\mathbb{Z}$ as an infinite version of a positive definite Toeplitz matrix.

\subsection{Fourier correspondence}

The correspondence between a positive definite function $C$ on $\mathbb{Z}$ and a finite Borel measure $\mu$ on $\mathbb{T}$ is given by the formula \begin{equation} \label{eq.fourier} C(n) = \hat{\mu}(n) = \int_\mathbb{T} s^{-n}\deee \mu(s) \end{equation} Thus the values of $C$ are the Fourier coefficients of $\mu$. The fact that the sequence of Fourier coefficients of a measure is positive definite reflects the following fact, which can be verified with elementary computation. \begin{align*} \sum_{m,n=-N}^N\alpha(m)\ov{\alpha(n)} C(m-n) &= \sum_{m,n=-N}^N\alpha(m)\ov{\alpha(n)}\int_\mathbb{T} s^{n-m} \deee\mu(s) \\ & = \int_\mathbb{T} \left \vert \sum_{m=-N}^N \alpha(m)s^{-m} \right \vert^2 \deee \mu(s) \geq 0 \end{align*} The fact that a positive definite function on $\mathbb{Z}$ uniquely defines a finite Borel measure on $\mathbb{T}$ via (\ref{eq.fourier}) is known as B\^{o}chner's theorem.

\subsection{Spectral correspondence}

The correspondence between a positive definite function $C$ on $\mathbb{Z}$ and a marked unitary representation $(\rho,x)$ of $\mathbb{Z}$ is given as follows. Let $u = \rho(1)$ be the unitary operator corresponding to the unique free generator of $\mathbb{Z}$. Then we set \begin{equation} \label{eq.matrixc} C(n) = \langle u^nx, x \rangle \end{equation} Thus the values of $C$ are the matrix coefficients of the marked unitary representation. The fact that a sequence of unitary matrix coefficients is positive definite reflects the following fact, which can be verified with elementary computation. \begin{align*} \sum_{m,n =-N}^N \alpha(m)\ov{\alpha(n)}C(m-n) &= \sum_{m,n=-N}^N \alpha(m)\ov{\alpha(n)} \langle u^{m-n}x, x \rangle \\ & = \nml \sum_{m=-N}^N \alpha(m)u^mx \nmr^2 \geq 0 \end{align*} The fact that a positive definite function on $\mathbb{Z}$ uniquely defines a marked unitary representation of $\mathbb{Z}$ via (\ref{eq.matrixc}) is a version of the spectral theorem.

\subsection{Constructing positive definite functions on the integers} \label{subsec.integers}

In Subsection \ref{subsec.integers} we describe a method for recursively constructing positive definite functions on the integers. A noncommutative version of this construction is the main topic of Section \ref{subsec.aoc}\\
\\
Note that if $m,n \in \{0,\ldots,N\}$ then we have $m-n \in \{-N,\ldots,N\}$. Suppose we have defined a function $C:\{-N,\ldots,N\} \to \mathbb{C}$ which satisfies the restricted fundamental inequality \begin{equation} \sum_{m,n=0}^N \alpha(m)\ov{\alpha(n)}C(m-n) \geq 0 \label{eq.fundamental-0.2} \end{equation} for every function $\alpha:\{0,\ldots,N\} \to \mathbb{C}$. We refer to $C$ as a partially defined positive definite function. We wish to extend $C$ to a partially defined positive definite function defined on $\{-N-1,\ldots,N+1\}$. In order to do so it suffices to specify the number $C(N+1)$, as then we must have $C(-N-1) = \ov{C(N+1)}$.\\
\\
We will assume that $C$ is strictly positive definite in the sense that the inequality in (\ref{eq.fundamental-0.2}) is saturated only when $\alpha$ is identically zero. Standard theory then implies that we can find a Hilbert space $\mathscr{X}$ of dimension $N+1$ and canonical basis vectors $\Phi_0,\ldots,\Phi_N \in \mathscr{X}$ such that \begin{equation} \label{eq.waiting-4} \langle \Phi_m,\Phi_n \rangle = C(m-n) \end{equation} for all $m,n \in \{0,\ldots,N\}$. We can also define a shifted Hilbert space $\mathscr{Y}$ of dimension $N+1$ with canonical basis vectors $\Phi_1,\ldots,\Phi_{N+1}$ satisfying (\ref{eq.waiting-4}) for all $m,n \in \{1,\ldots,N+1\}$. Let $\mathscr{Z}$ be the vector space consisting of the span of $\mathscr{X}$ and $\mathscr{Y}$, so that $\mathscr{Z}$ has a canonical basis $\Phi_0,\ldots,\Phi_{N+1}$.\\
\\
In the vector space $\mathscr{Z}$ the inner product between $\Phi_0$ and $\Phi_{N+1}$ is not defined. However, the inner products between all other pairs of elements of the canonical basis for $\mathscr{Z}$ are defined. This allows us to apply the Gram-Schmidt procedure to $\Phi_1,\ldots,\Phi_N$ to obtain an orthonormal basis $B$ for the subspace \begin{equation} \label{eq.subspace} \mathscr{X}  \cap  \mathscr{Y} = \mathrm{span}(\Phi_1,\ldots,\Phi_N) \end{equation} of $\mathscr{Z}$. Moreover, we can compute the inner products between $\Phi_0$ and the elements of $B$ and we can compute the inner products between $\Phi_{N+1}$ and the elements of $B$. Therefore the orthogonal projection $p$ from $\mathscr{Z}$ onto the subspace in (\ref{eq.subspace}) is well defined.\\
\\
Write $I$ for the identity operator on $\mathscr{Z}$. Then for any complex number $\zeta$ with $|\zeta| \leq 1$ we can set \begin{equation} \label{eq.waiting} \left \langle \frac{(I-p)\Phi_{N+1}}{||(I-p)\Phi_{N+1}||},\frac{(I-p)\Phi_0}{||(I-p)\Phi_0||} \right \rangle = \zeta \end{equation} The hypothesis that $C$ is strictly positive definite ensures the denominators of the fractions in (\ref{eq.waiting}) are nonzero. The hypothesis that $|\zeta| \leq 1$ reflects to the need to satisfy the Cauchy-Schwartz inequality. Once we have chosen $\zeta$, the space $\mathscr{Z}$ is promoted to a full Hilbert space. The fact that any choice of $\zeta$ with $|\zeta| \leq 1$ produces a valid positive definite extension follows from the observation that no matter the value of $\zeta$ we have a decomposition \begin{equation} \label{eq.oplus} \mathscr{Z} = (\mathscr{X} \cap \mathscr{Y}) \oplus \mathrm{span}\bigl((I-p)\Phi_{N+1},(I-p)\Phi_0\bigr) \end{equation} From (\ref{eq.waiting}) we can recover \begin{align} C(N+1) & = \langle \Phi_{N+1},\Phi_0\rangle \nonumber \\ & = \zeta ||(I-p)\Phi_{N+1}||\,||(I-p)\Phi_0|| + \langle p \Phi_{N+1},\Phi_0\rangle \nonumber \\ &\quad \quad + \langle \Phi_{N+1},p\Phi_0 \rangle - \langle p\Phi_{N+1},p\Phi_0\rangle \nonumber \\ & = \zeta ||(I-p)\Phi_{N+1}||\,||(I-p)\Phi_0|| + \langle p \Phi_{N+1},p\Phi_0 \rangle  \label{eq.waiting-2} \end{align} The norms and inner products in (\ref{eq.waiting-2}) are determined by $C$, so $\zeta$ is indeed the only free parameter. Thus the set of legal possibilities for $C(N+1)$ is a closed disk in $\mathbb{C}$ of radius $||(I-p)\Phi_{N+1}||\,||(I-p)\Phi_0||$ centered at the point $\langle p \Phi_{N+1},p\Phi_0 \rangle$. The numbers $||(I-p) \Psi_0||$ and  $||(I-p) \Phi_{N+1}||$ are bounded above by $||\Phi_0|| = ||\Phi_{N+1}|| = C(0)$. Moreover, these numbers will be small if the numbers $||p  \Phi_0||$ and $||p  \Phi_{N+1}||$ are close to $C(0)$. Since $p$ is the orthogonal projection onto $\mathscr{X} \cap \mathscr{Y} $, we can interpret this as indicating that if $\Phi_0$ and $\Phi_{N+1}$ are close to $\mathscr{X} \cap \mathscr{Y}$ then the possibilities for $C(N+1)$ are more restricted. The number $\zeta$ is typically referred to as a Szeg\"{o} parameter.

\section{Constructing positive definite functions on the free group} \label{subsec.aoc}

In Section \ref{subsec.aoc} we describe the general procedure for constructing positive definite functions on the free group.

\subsection{Geometry of generalized Cayley graphs on the free group} \label{seg.geom}

In Subsection \ref{seg.geom} we describe certain geometric features of the graphs $\mathrm{Cay}(\mathbb{F},g)$ which will be necessary for the construction of positive definite functions. We will use two known results about generalized Cayley graphs on $\mathbb{F}$. \\
\\
Recall that a graph is said to be chordal if every induced cycle has length at most $3$. The next fact appears as Proposition 3.2 in \cite{MR2316876} and as Proposition 3.6.7 in \cite{MR2807419}.

\begin{proposition}[Bakonyi, Timotin] \label{prop.chord} Let $g \in \mathbb{F}$. Then $\mathrm{Cay}(\mathbb{F},g)$ is chordal. \end{proposition}

Let $g \in \mathbb{F}$ and suppose $\mathcal{K}$ is a clique in $\mathrm{Cay}(\mathbb{F},g)$ which is not a clique in $\mathrm{Cay}(\mathbb{F},g_\uparrow)$. We note that $\mathcal{K}$ is maximal among all cliques in $\mathrm{Cay}(\mathbb{F},g)$ if and only if $\mathcal{K}$ is maximal among those cliques in $\mathrm{Cay}(\mathbb{F},g)$ which are not cliques in $\mathrm{Cay}(\mathbb{F},g_\uparrow)$. The following appears as Corollary 3.3 in \cite{MR2316876} and as Corollary 3.6.8 in \cite{MR2807419}.

\begin{proposition}[Bakonyi, Timotin] \label{prop.clique} Let $g \in \mathbb{F}$ and let $\mathcal{K}$ be a maximal clique in $\mathrm{Cay}(\mathbb{F},g)$ which is not a clique in $\mathrm{Cay}(\mathbb{F},g_\uparrow)$. Then there exists a unique edge in $\mathcal{K}$ which is not an edge in $\mathrm{Cay}(\mathbb{F},g_\uparrow)$. \end{proposition}

The next two propositions are implicit in \cite{MR2316876} but we include proofs for completeness.

\begin{proposition} \label{prop.max} Let $g \in \mathbb{F}$. Then there exists a unique maximal clique $\mathcal{K}_g$ in $\mathrm{Cay}(\mathbb{F},g)$ which contains the edge $(e,g)$.\end{proposition}

\begin{proof}[Proof of Proposition \ref{prop.max}] Suppose toward a contradiction that Proposition \ref{prop.max} fails. Let $\mathcal{M}_1$ and $\mathcal{M}_2$ be two distinct maximal cliques in $\mathrm{Cay}(\mathbb{F},g)$ which contain $(e,g)$. Since $\mathcal{M}_1$ and $\mathcal{M}_2$ contain $(e,g)$, they are not cliques in $\mathrm{Cay}(\mathbb{F},g_\uparrow)$. Therefore Proposition \ref{prop.clique} implies that for each $j \in \{1,2\}$ the edge $(e,g)$ is the unique edge in $\mathcal{M}_j$ which is not an edge in $\mathrm{Cay}(\mathbb{F},g_\uparrow)$. Thus for each $j \in \{1,2\}$ and every $m \in \mathcal{M}_j \setminus \{g,e\}$ we have that $(m,g)$ and $(m,e)$ are edges of $\mathrm{Cay}(\mathbb{F},g_\uparrow)$. \\
\\
Since $\mathcal{M}_1$ and $\mathcal{M}_2$ are distinct and maximal, it must be the case that their union is not a clique in $\mathrm{Cay}(\mathbb{F},g)$. Therefore we can choose $m_1 \in \mathcal{M}_1 \setminus \mathcal{M}_2$ and $m_2$ in $\mathcal{M}_2 \setminus \mathcal{M}_1$ such that $(m_1,m_2)$ is not an edge in $\mathrm{Cay}(\mathbb{F},g)$. Since $\{g,e\} \subseteq \mathcal{M}_1 \cap \mathcal{M}_2$ for each $j \in \{1,2\}$ we have $m_j \notin \{g,e\}$. Again using the fact that $(g,e)$ is the unique edge added to $\mathcal{M}_1$ and $\mathcal{M}_2$ when passing from $\mathrm{Cay}(\mathbb{F},g_\uparrow)$ to $\mathrm{Cay}(\mathbb{F},g)$ we see that $(m_1,m_2)$ is not an edge in $\mathrm{Cay}(\mathbb{F},g_\uparrow)$.\\
\\
Consider the path $e \to m_1 \to g \to m_2 \to e$. The previous paragraphs show that this is an induced cycle in $\mathrm{Cay}(\mathbb{F},g_\uparrow)$, so we obtain a contradiction to Proposition \ref{prop.chord}. \end{proof}

\begin{proposition} \label{prop.translate} Suppose $\mathcal{M}$ is a maximal clique in $\mathrm{Cay}(\mathbb{F},g)$ which is not a clique in $\mathrm{Cay}(\mathbb{F},g_\uparrow)$. Then $\mathcal{M}$ is a translate of $\mathcal{K}_g$. \end{proposition}

\begin{proof}[Proof of Proposition \ref{prop.translate}] Since $\mathcal{M}$ is not a clique in $\mathrm{Cay}(\mathbb{F},g_\uparrow)$ it must contain a pair of vertexes $h,\ell$ such that $\ell^{-1}h \notin \mathcal{I}_{g_\uparrow}$. Since $g$ is the unique element of $\mathcal{I}_g \setminus \mathcal{I}_{g_\uparrow}$ and $\mathcal{M}$ is a clique in $\mathrm{Cay}(\mathbb{F},g)$ this implies that $\ell^{-1}h = g$. Thus by translating we may assume that $h=g$ and $\ell = e$ and so Proposition \ref{prop.translate} follows from Proposition \ref{prop.max}. \end{proof}

The next proposition is necessary because it may happen that $g_\uparrow \notin \mathcal{K}_g$.

\begin{proposition} \label{prop.weird} Let $g \in \mathbb{F}$. Then there exists $h \in \mathcal{I}_g$ such that $\mathcal{K}_g \setminus \{g\}$ is contained in a translate of $\mathcal{K}_h$. \end{proposition}

\begin{proof}[Proof of Proposition \ref{prop.weird}] Proposition \ref{prop.chord} implies that $(e,g)$ is the unique edge in $\mathcal{K}_g$ which is not an edge in $\mathrm{Cay}(\mathbb{F},g_\uparrow)$. Therefore we have that $\mathcal{K}_g \setminus \{ g\}$ is a clique in $\mathrm{Cay}(\mathbb{F},g_\uparrow)$. Let $h$ be the $\preceq$-least element such that $\mathcal{K}_g \setminus \{g\}$ is a clique in $\mathrm{Cay}(\mathbb{F},h)$. Let $\mathcal{M}$ be a maximal clique in $\mathrm{Cay}(\mathbb{F},h)$ which contains $\mathcal{K}_g \setminus \{g\}$. If $\mathcal{M}$ were a clique in $\mathrm{Cay}(\mathbb{F},h_\uparrow)$ then $\mathcal{K}_g \setminus \{g\}$ would be a clique in $\mathrm{Cay}(\mathbb{F},h_\uparrow)$, contradicting our choice of $h$. Thus Proposition \ref{prop.max} implies $\mathcal{M}$ is contained in a translate of $\mathcal{K}_h$. \end{proof}

\subsection{Full realizations of partial positive definite functions} \label{seg.partreal}

Write $\mathrm{NSPD}_g$ for the space of normalized strictly positive definite functions from $\mathcal{I}_g$ to $\mathbb{C}$. We refer to such a function as a partial positive definite function. Let $\mathsf{C} \in \mathrm{NSPD}_g$. If $\mathcal{J}$ is a clique in $\mathrm{Cay}(\mathbb{F},g)$, there exists a Hilbert space $\mathscr{X}(\mathsf{C},\mathcal{J})$ and a function $\Phi_\mathsf{C}:\mathcal{J} \to \mathscr{X}(\mathsf{C},\mathcal{J})$ such that \begin{equation} \label{eq.blade-1} \langle \Phi_\mathsf{C}(h),\Phi_\mathsf{C}(\ell) \rangle = \mathsf{C}(\ell^{-1}h) \end{equation} for all $h,\ell \in \mathcal{J}$. The assertion that $\mathcal{J}$ is a clique in $\mathrm{Cay}(\mathbb{F},g)$ is equivalent to the assertion that $\mathcal{J}^{-1}\mathcal{J} \subseteq \mathcal{I}_g$. The relevance of this hypothesis is that it ensures the inner products between all pairs of elements of the range of $\Phi$ are defined. If we require that the set  $\{\Phi_\mathsf{C}(h): h \in \mathcal{J}\}$ spans $\mathscr{X}(\mathsf{C},\mathcal{J})$ then these data are unique up to a unique unitary isomorphism. If $\mathsf{C} \in \mathrm{NSPD}_g$ write $\mathscr{X}(\mathsf{C})$ for $\mathscr{X}(\mathsf{C},\mathcal{K}_g)$.

\subsection{Partial realizations of partial positive definite functions}

The following definition makes precise the construction of the space $\mathscr{Z}$ from Subsection \ref{subsec.integers}.

\begin{definition} \label{def.partial} A \textbf{partial Hilbert space} is a vector space $\mathscr{V}$ together with a pair of distinguished subspaces $\mathscr{V}_1$ and $\mathscr{V}_2$ of $\mathscr{V}$ having the following properties. \begin{itemize} \item The subspaces $\mathscr{V}_1$ and $\mathscr{V}_2$ span $\mathscr{V}$. \item Each $\mathscr{V}_m$ is a Hilbert space with inner product $\langle \cdot, \cdot \rangle_m$. \item The inner products on $\mathscr{V}_1$ and $\mathscr{V}_2$ are compatible in the sense that for any pair of vectors $x,y \in \mathscr{V}_1 \cap \mathscr{V}_2$ we have $\langle x,y \rangle_1 = \langle x,y \rangle_2$. \end{itemize} We refer to $\mathscr{V}_1 \cap \mathscr{V}_2$ as the \textbf{core} of $\mathscr{V}$ and denote it by $\mathrm{core}(\mathscr{V})$. \end{definition}

The following proposition is the key to the extension procedure.

\begin{proposition} If $\mathscr{V}$ is a partial Hilbert space then the orthogonal projection from $\mathscr{V}$ onto $\mathrm{core}(\mathscr{V})$ is well-defined. \label{prop.ortho} \end{proposition}

\begin{proof}[Proof of Proposition \ref{prop.ortho}] The third item in Definition \ref{def.partial} ensures that we can apply the Gram-Schmidt procedure to obtain an orthonormal basis for $\mathrm{core}(\mathscr{V})$. The first and second items in Definition \ref{def.partial} imply that we can compute the inner products between an arbitrary element of $\mathscr{V}$ and and element of $\mathrm{core}(\mathscr{V})$. Therefore Proposition \ref{prop.ortho} follows. \end{proof}

We now describe how to associate a partial Hilbert space to a partial positive definite function. This partial Hilbert space will be used to calculate the set of legal possibilities for an extension of an element of $\mathrm{NSPD}_{g_\uparrow}$ to an element of $\mathrm{NSPD}_g$.\\
\\
Let $\mathsf{C} \in \mathrm{NSPD}_{g_\uparrow}$. Proposition \ref{prop.clique} implies the clique $\mathcal{K}_g$ in $\mathrm{Cay}(\mathbb{F},g)$ contains a unique edge which is not an edge in $\mathrm{Cay}(\mathbb{F},g_\uparrow)$. This must be the edge between $g$ and $e$. Therefore the sets $\mathcal{K}_{g} \setminus \{g\}$ and $\mathcal{K}_{g} \setminus \{e\}$ are cliques in $\mathrm{Cay}(\mathbb{F},g_\uparrow)$. Hence the discussion in Segment \ref{seg.partreal} implies that we can construct two Hilbert spaces $\mathscr{X}(\mathsf{C},\mathcal{K}_g \setminus \{g\})$ and $\mathscr{X}(\mathsf{C},\mathcal{K}_g \setminus \{e\})$ together with functions \[ \Phi_\mathsf{C}:\mathcal{K}_{g} \setminus \{g\} \to \mathscr{X}(\mathsf{C},\mathcal{K}_g \setminus \{g\}) \] and \[ \Psi_\mathsf{C}: \mathcal{K}_{g} \setminus \{e\} \to \mathscr{X}(\mathsf{C},\mathcal{K}_g \setminus \{e\}) \] such that $\Phi_\mathsf{C}$ and $\Psi_\mathsf{C}$ satisfy (\ref{eq.blade-1}) on their domains. 

\begin{proposition} Let $\mathsf{C} \in \mathrm{NSPD}_{g_\uparrow}$. Then there exists a partial Hilbert space $\mathscr{X}(\mathsf{C})_\bullet$ with distinguished subspaces that can be identified with $\mathscr{X}(\mathsf{C},\mathcal{K}_g \setminus \{g\})$ and $\mathscr{X}(\mathsf{C},\mathcal{K}_g \setminus \{e\})$. Moreover, if $h \in \mathcal{K}_g \setminus \{g,e\}$ then $\Psi_\mathsf{C}(h) = \Phi_\mathsf{C}(h)$ and $\mathrm{core}(\mathscr{X}(\mathsf{C})_\bullet)$ consists exactly of the span of these vectors. \label{prop.partial} \end{proposition}

\begin{proof}[Proof of Proposition \ref{prop.partial}] We take $\mathscr{X}(\mathsf{C})_\bullet$ to be the quotient of the disjoint union of $\mathscr{X}(\mathsf{C},\mathcal{K}_g \setminus \{g\})$ and  $\mathscr{X}(\mathsf{C},\mathcal{K}_g \setminus \{e\})$ by the equivalence relation which identifies $\Phi_\mathsf{C}(h)$ and $\Psi_\mathsf{C}(h)$ for $h \in \mathcal{K}_g \setminus \{g,e\}$. \end{proof}

If $\mathsf{C} \in \mathrm{NSPD}_{g_\uparrow}$ we define \begin{equation} \label{eq.lewis-1} \Theta_\mathsf{C}(h) = \begin{cases} \Phi_\mathsf{C}(h) = \Psi_\mathsf{C}(h) \mbox{ if }h \in \mathcal{K}_g \setminus \{g,e\} \\ \Psi_\mathsf{C}(h)\mbox{ if }h=g \\ \Phi_\mathsf{C}(h) \mbox{ if }h=e \end{cases} \end{equation} Thus $\{\Theta_\mathsf{C}(h):h \in \mathcal{K}_g \}$ forms a canonical basis for $\mathscr{X}(\mathsf{C})_\bullet$.

\subsection{Parameterizing extensions} \label{subsec.extension-def}

In Segment \ref{subsec.extension-def} we describe the procedure for extending positive definite functions. This construction has its roots in \cite{MR2316876} and in Section 3.6 of \cite{MR2807419}.\\
\\
Let $g \in \mathbb{F}$ and fix $\mathsf{C} \in \mathrm{NSPD}_{g_\uparrow}$. We wish to understand extensions of $\mathsf{C}$ to an element of $\mathrm{NSPD}_g$. In order to describe such an extension it suffices to specify the number $\mathsf{C}(g)$ since then we must have $\mathsf{C}(g^{-1}) = \ov{\mathsf{C}(g)}$.\\
\\
We begin by constructing the partial Hilbert space $\mathscr{X}(\mathsf{C})_\bullet$ as in Proposition \ref{prop.partial}. Specifying the matrix $\mathsf{C}(g)$ amounts to specifying the inner product between $\Theta_\mathsf{C}(g)$ and $\Theta_\mathsf{C}(e)$. Write $I$ for the identity operator on $\mathscr{X}(\mathsf{C})_\bullet$. By Proposition \ref{prop.ortho} we can consider the orthogonal projection $p$ from $\mathscr{X}(\mathsf{C})_\bullet$ onto $\mathrm{core}(\mathscr{X}(\mathsf{C})_\bullet)$. Given a complex number $\zeta$ with $|\zeta| < 1$, we can make the following definition, which generalizes (\ref{eq.waiting}). Set \begin{equation} \label{eq.done-1} \left \langle \frac{(I-p)  \Theta_\mathsf{C}(g)}{||(I-p)  \Theta_\mathsf{C}(g)||}, \frac{(I-p)  \Theta_\mathsf{C}(e)}{||(I-p)  \Theta_\mathsf{C}(e)||} \right \rangle = \zeta \end{equation}

The hypothesis that $\mathsf{C}$ is strictly positive definite implies that the denominators of the fractions in (\ref{eq.done-1}) are nonzero. As with (\ref{eq.waiting}), the requirement that $|\zeta| \leq 1$ is immediate from the need to satisfy the Cauchy-Schwartz inequality. The requirement that $|\zeta| \neq 1$ will be discussed in Segment \ref{subsec.done}. From (\ref{eq.done-1}) we have the analog of (\ref{eq.waiting-2}), whereby we recover $\mathsf{C}(g)$ as \begin{align} \mathsf{C}(g)& = \langle \Theta_\mathsf{C}(g),\Theta_\mathsf{C}(e) \rangle \nonumber \\ & = \zeta ||(I-p) \Theta_\mathsf{C}(g)||\,||(I-p) \Theta_\mathsf{C}(e)|| + \langle p  \Theta_\mathsf{C}(g),p\Theta_\mathsf{C}(e) \rangle \label{eq.hey-0} \end{align}

The norms and inner product in (\ref{eq.hey-0}) are determined by $\mathsf{C}$, so that $\zeta$ is indeed the only free parameter. The fact that any value of $\zeta$ with $|\zeta| \leq 1$ produces a valid positive definite function follows by the analog of (\ref{eq.oplus}), which is the decomposition

\[ \mathscr{X}(\mathsf{C})_\bullet = \mathrm{core}(\mathscr{X}(\mathsf{C})_\bullet) \oplus \mathrm{span}\bigl((I-p)\Theta_\mathsf{C}(g),(I-p)\Theta_\mathsf{C}(e)\bigr) \] valid for any value of $\zeta$. Write $\mathsf{C}^\zeta$ for the extension of $\mathsf{C}$ by $\zeta$. Thus after extending by $\zeta$ the partial Hilbert space $\mathscr{X}(\mathsf{C})_\bullet$ is promoted to a full Hilbert space $\mathscr{X}(\mathsf{C}^\zeta)$. The Hilbert space $\mathscr{X}(\mathsf{C}^\zeta)$ has a canonical basis indexed by $\mathcal{K}_g$ and we can regard $\mathsf{C}^\zeta$ as an element of $\mathrm{NSPD}_g$. It is natural to think of $\zeta$ as a noncommutative Szeg\"{o} parameter. \\
\\
If we choose $\zeta = 0$ at every step of the extension procedure we obtain the so-called `central extension', which corresponds to the construction of a higher-step Markov process on the free group. This is the construction given in Lemma 24 of \cite{MR3067294}. The central extension does not have the properties required to prove Theorem \ref{thm} and we must choose the extension parameters for $\mathsf{C}$ more carefully. However, in Theorem \ref{thm} we can take $\hat{\mathsf{D}}$ to be the central extension of $\mathsf{D}$. \\

\subsection{Transport operators of partially defined functions} \label{subsec.transport}

We will need the following analog of Definition \ref{def.transportop}.

\begin{definition} \label{def.transport} Let $g \in \mathbb{F}$ and let $\mathsf{C},\mathsf{D} \in \mathrm{NSPD}_g$. The \textbf{transport operator} between the partial Hilbert spaces $\mathscr{X}(\mathsf{C})_\bullet$ and $\mathscr{X}(\mathsf{D})_\bullet$ is denoted $t[\mathsf{C},\mathsf{D}]$ and is given by setting \[ t[\mathsf{C},\mathsf{D}]  \sum_{h \in \mathcal{K}_g} \alpha(h) \Theta_\mathsf{C}(h) = \sum_{h \in \mathcal{K}_g} \alpha(h) \Theta_\mathsf{D}(h) \]  for $\alpha:\mathcal{K}_g \to \mathbb{C}$.  We define the \textbf{relative energy} of the pair $(\mathsf{C},\mathsf{D})$ to be the maximum of the squares of the operator norms of the restrictions of $t[\mathsf{C},\mathsf{D}]$ to the distinguished subspaces $\mathscr{X}(\mathsf{C})_g$ and $\mathscr{X}(\mathsf{C})_e$. We denote the relative energy of the pair $(\mathsf{C},\mathsf{D})$ by $\mathfrak{e}(\mathsf{C},\mathsf{D})$. \end{definition}

There is a slight difference between Definitions \ref{def.transportop} and \ref{def.transport}. If $|g| = 2r$ and $g$ is the $\preceq$ last element of its length then $\mathcal{K}_g$ is a translate of $\mathbb{B}_r$ different from $\mathbb{B}_r$.

\subsection{Degenerate extensions} \label{subsec.done}

Observe that the construction described in Segment \ref{subsec.extension-def} makes sense if we choose $\zeta$ to be an element of the unit circle. However, in this case the resulting extension will not be strictly positive definite and so the denominators of the fractions in (\ref{eq.done-1}) will be zero at some later stage of the procedure. Thus we regard $|\zeta| = 1$ as an unacceptable degeneracy. Intuitively, such a degeneracy corresponds to a cycle in the time evolution represented by the extension procedure. The following definition will allow us to avoid this issue in the context of minimizing relative energies.

\begin{definition} \label{def.singdeg} Let $g \in \mathbb{F}$ and let $\mathsf{C},\mathsf{D} \in \mathrm{NSPD}_g$. We define the pair $(\mathsf{C},\mathsf{D})$ to have \textbf{singular degeneracies} if there is a constant $c(\mathsf{C},\mathsf{D})>0$ such that \[ \mathfrak{e}(\mathsf{C}^\zeta,\mathsf{D}^\mu) \geq \frac{c(\mathsf{C},\mathsf{D})}{1-|\zeta|^2} \] for all $\zeta,\mu \in \mathbb{D}$. \end{definition}

We now formulate a lemma that will be used in the proof of Theorem \ref{thm}. We will prove it below in Section \ref{segment.timeloop}. 

\begin{lemma}[Small perturbations give singular degeneracies] \label{lem.singular} Let $g \in \mathbb{F}$ be such that $|g| \geq 5$ and let $\epsilon >0$. Let also $\mathsf{C}_1,\mathsf{C}_2 \in \mathrm{NSPD}_g$. Then there exist $\underline{\mathsf{C}}_1,\underline{\mathsf{C}}_2 \in \mathrm{NSPD}_g$ such that \[ \max_{j \in \{1,2\}}||\mathsf{C}_j - \underline{\mathsf{C}}_j||_1 \leq \epsilon\] and such the pair $(\underline{\mathsf{C}}_1,\underline{\mathsf{C}}_2)$ has singular degeneracies. \end{lemma}

\section{Proof of Lemma \ref{lem.singular}} \label{segment.timeloop}

In Section \ref{segment.timeloop} we prove Lemma \ref{lem.singular}.

\subsection{Smoothness of Gram-Schmidt} \label{seg.gram}

In Segment \ref{seg.gram} we establish Proposition \ref{prop.star}. This is a general result about the Gram-Schmidt procedure which is likely well-known. 

\begin{definition} \label{def.star} Let $n \in \mathbb{N}$ and let $M \in \mathrm{Mat}_{n \times n}(\mathbb{C})$ be strictly positive definite with ones on the diagonal. Let $\mathscr{Z}$ be a Hilbert space and let $y_1,\ldots,y_n \in \mathscr{Z}$ be a basis such that $\langle y_j,y_k \rangle= M_{j,k}$ for all $j,k \in [n]$. Let $z_1,\ldots,z_n$ be the orthogonal basis obtained by applying the Gram-Schmidt procedure to $y_1,\ldots,y_n$. We define the \textbf{orthogonalization matrix} of $M$ to be the $n \times n$ matrix which changes $y_1,\ldots,y_n$ coordinates to $z_1,\ldots,z_n$ coordinates. We denote the orthogonalization matrix by $\mathcal{G}(M)$. Also define the \textbf{orthonormalization matrix} of $M$ to be the matrix which changes $y_1,\ldots,y_n$ coordinates to $z_1||z_1||^{-1},\ldots,z_n||z_n||^{-1}$ coordinates. We denote the orthonormalization matrix by $\mathcal{N}(M)$.  \end{definition}

\begin{proposition} \label{prop.star} The entries of $\mathcal{G}(M)$ are analytic functions of the entries of $M$. The entries of $\mathcal{N}(M)$ are differentiable functions of $M$. \ \end{proposition}

\begin{proof}[Proof of Proposition \ref{prop.star}] We establish the proposition by induction on $n$. The case $n=1$ is trivial, so assume we have established the result for $n$ with the goal of establishing it for $n+1$.  Fix an $n \times n$ strictly positive definite matrix $M_\bullet$ with ones on the diagonal. If $M$ is an $(n +1) \times (n+1)$ matrix we have that the first $n$ columns of $\mathcal{G}(M)$ depend only on the $n \times n$ upper left corner of $M$. Therefore we will express $M$ in terms of a variable $\mathbf{x} = (x_1,\ldots,x_n,1) \in \mathbb{C}^{n+1}$ representing a column to be augmented on the right of $M_\bullet$ to form a $(n+1) \times (n+1)$ strictly positive definite matrix $M(\mathbf{x})$. Note that for a given strictly positive definite $M_\bullet$ the set of $\mathbf{x} \in \mathbb{C}^{n+1}$ such that $M(\mathbf{x})$ is strictly positive definite is open. Moreover, it is nonempty since it contains the vector $(0,\ldots,0,1)$. \\
\\
Let $\mathbf{q}(\mathbf{x}) = (q_1(\mathbf{x}),\ldots,q_n(\mathbf{x}),1)$ be the $k+1$ column of $\mathcal{G}(M(\mathbf{x}))$. For $k \in [n]$ let $M_\bullet(\check{k})$ be $M_\bullet$ with the $k^{\mathrm{th}}$ column removed. We have the following expression for $\mathbf{q}(\mathbf{x})$.

\begin{equation} \label{eq.star-5} q_k(\mathbf{x}) = (-1)^{k+n}\frac{\mathrm{det}(M_\bullet(\check{k}) \,| \mathbf{x})}{\mathrm{det}(M_\bullet)}. \end{equation} This appears, for example, as (35) in Section 6 of Chapter $\mathrm{IX}$ of \cite{MR869996}. The first clause in Proposition \ref{prop.star} is clear from (\ref{eq.star-5}). Writing $(r_1(\mathbf{x}),\ldots,r_n(\mathbf{x}),r_{n+1}(\mathbf{x}))$ for the $k+1$ column of $\mathcal{N}(M(\mathbf{x}))$ from the same reference we have the formula \begin{equation} \label{eq.star-88} r_k(\mathbf{x}) = (-1)^{k+n}\frac{\mathrm{det}(M_\bullet(\check{k}) \,| \mathbf{x})}{\sqrt{\mathrm{det}(M_\bullet)\mathrm{det}(M(\mathbf{x}))}}. \end{equation} The second clause in Proposition \ref{prop.star} is clear from (\ref{eq.star-88}).  \end{proof}

\subsection{First perturbation of the configuration} \label{seg.1pert}

Let $g \in \mathbb{F}$, let $\epsilon > 0$ and let $\mathsf{C}_1,\mathsf{C}_2 \in \mathrm{NSPD}_g$. For $s \in \{1,\ldots,|g|\}$ let $g_s$ be the word consisting of the first $s$ letters of $g$. Given a $4$-tuple $\Lambda= (\lambda_1,\lambda_2,\lambda_3,\lambda_4)$ of complex numbers and $j \in \{1,2\}$, let $\mathsf{C}_{j,\Lambda}$ denote the modification of $\mathsf{C}_j$ given by setting 

\begin{align*}  \mathsf{C}_{j,\Lambda}(g_1^{-1}g) & =  \ov{\mathsf{C}_{j,\Lambda}(g^{-1}g_1)}   = \mathsf{C}_j(g_1^{-1}g)+\lambda_1 \\ \mathsf{C}_{j,\Lambda}(g_2^{-1}g) &= \ov{\mathsf{C}_{j,\Lambda}(g^{-1}g_2)} =  \mathsf{C}_j(g_2^{-1}g)+\lambda_2 \\ \mathsf{C}_{j,\Lambda}(g_1^{-1}) &= \ov{\mathsf{C}_{j,\Lambda}(g_1)} =  \mathsf{C}_j(g_1^{-1})+\lambda_3 \\ \mathsf{C}_{j,\Lambda}(g_2^{-1}) & =\ov{\mathsf{C}_{j,\Lambda}(g_2)} =  \mathsf{C}_j(g_2^{-1})+\lambda_4 \end{align*} and leaving all other entries of $\mathsf{C}_j$ unchanged. Since the space of strictly positive definite functions is open, if $||\Lambda||_1$ is small enough we will have $\mathsf{C}_{j,\Lambda}\in \mathrm{NSPD}_g$.\\
\\
Let $w_{j,\Lambda}$ be the orthogonal projection from $\mathscr{X}(\mathsf{C}_{j,\Lambda})_\bullet$ onto the span of $\Theta_{\mathsf{C}_{j,\Lambda}}(g_1)$ and $\Theta_{\mathsf{C}_{j,\Lambda}}(g_2)$. Let $W_{j,\lambda}$ be the matrix of $w_{j,\Lambda}$ constructed with respect to the canonical basis $\{\Theta_{\mathsf{C}_{j,\Lambda}}(h): h \in \mathcal{K}_g \}$.\\
 \\
Note that this matrix does not depend on the correlation between $\Theta_{\mathsf{C}_j}(g)$ and $\Theta_{\mathsf{C}_j}(e)$ and so it is well-defined on the partial Hilbert space $\mathscr{X}(\mathsf{C}_{j,\Lambda})_\bullet$. Let $W'_{j,\Lambda}$ be the restriction of $W_{j,\Lambda}$ to the span of $\{\Theta_{\mathsf{C}_{j,\Lambda}}(g), \Theta_{\mathsf{C}_{j,\Lambda}}(e)\}$. Thus $W'_{j,\Lambda}$ is a $2 \times 2$ square matrix. Since $W'_{j,\Lambda}$ is an injective affine function of $\Lambda$ we can fix $\Lambda_j$ for each $j \in \{1,2\}$ such that $\min_{j \in \{1,2\}} |\det(W'_{j,\Lambda_j})| > 0$ and such that each $||\Lambda_j||_1$ small enough that $\mathsf{C}_{j,\Lambda} \in \mathrm{NSPD}_g$ and \[ \max_{j \in \{1,2\}} ||\mathsf{C}_j - \mathsf{C}_{j,\Lambda_j}||_1 \leq \frac{\epsilon}{2} \] Write $\tilde{\mathsf{C}}_j$ for this $\mathsf{C}_{j,\Lambda_j}$.

\subsection{Second perturbation of the configuration}

Let $\Delta \in \mathrm{NSPD}_g$ be given by setting $\Delta(e) = 1$ and letting all other values of $\Delta$ be equal to $0$. For $s \in [0,1]$ and $j \in \{1,2\}$ let $\tilde{\mathsf{C}}_{j,s} = (1-s)\tilde{\mathsf{C}}_j + s \Delta$.\\
\\
We specify that when performing the Gram-Schmidt orthogonalization procedure on the canonical basis $\{\Theta_{\tilde{\mathsf{C}}_{j,s}}(h):h \in \mathcal{K}_g\}$ the vectors $\{\Theta_{\tilde{\mathsf{C}}_{j,s}}(g),\Theta_{\tilde{\mathsf{C}}_{j,s}}(e) \}$ should be the last to be orthogonalized. Let $A_{j,s} = (\mathbf{I} \oplus \mathbf{0}) \mathcal{G}(\tilde{\mathsf{C}}_{j,s})$, where $\mathbf{I}$ is a copy of the identity matrix corresponding to the indexes in $\mathcal{K}_g \setminus \{g,e\}$ and $\mathbf{0}$ is a copy of the zero matrix corresponding to the indexes $\{g,e\}$. Note that since the last two rows of $A_{j,s}$ are zero, this matrix is well-defined even though the inner products $\langle \Theta_{\tilde{\mathsf{C}}_{j,s}}(g),\Theta_{\tilde{\mathsf{C}}_{j,s} }(e) \rangle$ are not yet specified. \\
\\
Proposition \ref{prop.star} implies that the entries of $A_{j,s}$ are real analytic functions of $s$. The kernel of $A_{j,s}$ is equal to the orthogonal complement of $\mathrm{core}(\mathscr{X}(\tilde{\mathsf{C}}_{j,s})_\bullet)$ in $\mathscr{X}(\tilde{\mathsf{C}}_{j,s})_\bullet$. Let $Q_{j,s}$ be the span of $\{\Theta_{\tilde{\mathsf{C}}_{j,s}}(g),\Theta_{\tilde{\mathsf{C}}_{j,s}}(e) \}$. Invertibility of the matrices $W'_{j,\Lambda_j}$ implies that orthogonal complement of $\mathrm{core}(\mathscr{X}(\tilde{\mathsf{C}}_{j,s})_\bullet)$ in $\mathscr{X}(\tilde{\mathsf{C}}_{j,s})_\bullet$ has trivial intersection with $Q_{j,s}$. Therefore the kernel of the matrix $A_{j,0}$ has trivial intersection with vectors supported on $\{g,e\}$. On the other hand, it is clear that the kernel of $A_{j,1}$ is equal to the vectors supported on $\{g,e\}$. Hence if we write \[ D(s_1,s_2) = \mathrm{det}\bigl(A_{1,s_1}^\ast A_{1,s_1} + A_{2,s_2}^\ast A_{2,s_2}\bigr) \] then $D(0,1) > 0$. Since $D(s_1,s_2)$ is a real analytic function of $s_1$ and $s_2$ we see that the set of pairs $(s_1,s_2) \in [0,1]^2$ such that $D(s_1,s_2) = 0$ has Lebesgue measure $0$. Hence we can choose $s_1,s_2$ such that $D(s_1,s_2) \neq 0$ and such that \[ \max_{j \in \{1,2\}} || \tilde{\mathsf{C}}_j - \tilde{\mathsf{C}}_{j,s_j}||_1 \leq \frac{\epsilon}{2} \]

Write $\underline{\mathsf{C}}_j$ for $\tilde{\mathsf{C}}_{j,s_j}$ and write $\tilde{A}_j$ for $A_{j,s_j}$. Since $D(s_1,s_2) > 0$ we see that the kernel of $\tilde{A}_1$ has trivial intersection with the kernel of $\tilde{A}_2$.Let $\theta > 0$ be such that if $\alpha:\mathcal{K}_g \to \mathbb{C}$ satisfies $||\alpha||_2 \geq 1$ and $\tilde{A}_1\alpha = 0$ then $||\tilde{A}_2\alpha||_2 \geq \theta$. Let $B$ be the orthogonal basis for $\mathrm{core}(\mathscr{X}(\underline{\mathsf{C}}_j)_\bullet)$ obtained by applying the Gram-Schmidt orthogonalization procedure to $\{ \Theta_{\underline{\mathsf{C}}_2}(h):h \in \mathcal{K}_g \setminus \{g,e\} \}$. Let $\kappa$ be the minimal norm of among all elements of $B$.

\subsection{Establishing the existence of energy singularities}

Fix $\zeta,\mu \in \mathbb{D}$ and consider the extensions $\underline{\mathsf{C}}_1^\zeta,\underline{\mathsf{C}}_2^\mu \in \mathrm{NSPD}_g$. Let $F$ be the orthogonal basis produced obtained applying the Gram-Schmidt orthogonalization procedure to the canonical basis $\{\Theta_{\underline{\mathsf{C}}_2^\mu}(h):h \in \mathcal{K}_g\}$. Note that this orthogonalization procedure can be completed in its entirety because we have specified $\langle \Theta_{\underline{\mathsf{C}}_2^\mu}(g),\Theta_{\underline{\mathsf{C}}_2^\mu}(e) \rangle = \mu$. Moreover, we stipulate that when performing this procedure the vectors $\Theta_{\underline{\mathsf{C}}^\mu_2}(g)$ and $\Theta_{\underline{\mathsf{C}}_2^\mu}(e)$ are orthogonalized at the last stage so that $F$ extends $B$.  \\
\\
Consider a function $\alpha:\mathcal{K}_g \to \mathbb{C}$, which defines a vector \[ x = \sum_{h \in \mathcal{K}_g} \alpha(h) \Theta_{\underline{\mathsf{C}}_2^\mu}(h) \] in the space $\mathscr{X}(\underline{\mathsf{C}}_2^\mu)$. If we rewrite $x$ in the orthogonal basis $F$, the resulting coordinates are given by $\tilde{A}_2\alpha$. Therefore $||x|| \geq \kappa ||\tilde{A}_l\alpha||_2$. From our choice of $\theta$ we see that if $\tilde{A}_1\alpha = 0$ then $||\tilde{A}_2\alpha||_2 \geq \theta||\alpha||_2$ so that $||x|| \geq \kappa \theta ||\alpha||_2$.\\
\\
Let $p$ be the orthogonal projection from the extended Hilbert space $\mathscr{X}(\underline{\mathsf{C}}_1^\zeta)$ onto $\mathrm{core}(\mathscr{X}(\underline{\mathsf{C}}_1)_\bullet)$. Then we have \begin{equation} \left \langle\ov{\zeta} \frac{(I-p)  \Theta_{\underline{\mathsf{C}}_1^\zeta}(g)}{||(I-p)  \Theta_{\underline{\mathsf{C}}_1^\zeta}(g)||},  \frac{(I-p)  \Theta_{\underline{\mathsf{C}}_1^\zeta}(e)}{||(I-p)  \Theta_{\underline{\mathsf{C}}_1^\zeta}(e)||} \right \rangle = |\zeta|^2 \label{eq.latin-3} \end{equation} 

Let \[ y  = \ov{\zeta} \frac{(I-p)  \Theta_{\underline{\mathsf{C}}_1^\zeta}(g)}{||(I-p)  \Theta_{\underline{\mathsf{C}}_1^\zeta}(g)||} - \frac{(I-p)  \Theta_{\underline{\mathsf{C}}_1^\zeta}(e)}{||(I-p)  \Theta_{\underline{\mathsf{C}}_1^\zeta}(e)||}  \] If we rewrite \[ y = \sum_{h \in \mathcal{K}} \alpha(h) \Theta_{\underline{\mathsf{C}}_1^\zeta}(h) \]

then we must have \[ \alpha(e) = \frac{1}{||(I-p)  \Theta_{\underline{\mathsf{C}}_1^\zeta}(e)||} \] and \[ \alpha(g) = \frac{1}{||(I-p)  \Theta_{\underline{\mathsf{C}}_1^\zeta}(g)||} \]  so that $||\alpha||_2 \geq 1$. Moreover, we have $\tilde{A}_1\alpha = 0$ since $y$ lies in the orthogonal complement of $\mathrm{core}(\mathscr{X}(\underline{\mathsf{C}}_1)_\bullet)$ in $\mathscr{X}(\underline{\mathsf{C}}_1^\zeta)$. It follows that $|| t[\underline{\mathsf{C}}_1^\zeta, \underline{\mathsf{C}}_2^\mu] y || \geq \kappa \theta$.\\
\\
On the other hand, from (\ref{eq.latin-3}) we see $||y||^2 \leq 2-2|\zeta|^2$. Therefore $\mathfrak{e}(\underline{\mathsf{C}}_1^\zeta,\underline{\mathsf{C}}_1^\mu) \geq \kappa^2 \theta^2(2-2|\zeta|^2)^{-1}$. Since $\kappa$ and $\theta$ are determined by $\mathsf{C}_1$ and $\mathsf{C}_2$ this completes the proof of Lemma \ref{lem.singular}.

\section{Proof of Theorem \ref{thm}} \label{seg.perturb}

In Subsections \ref{seg.inin} - \ref{seg.before} of Section \ref{seg.perturb} we establish Propositions \ref{prop.nonzero} - \ref{prop.anita-3}, which will be used below in Subsections \ref{oo} and \ref{ooo} to prove Theorem \ref{thm}.

\subsection{Introducing initial data} \label{seg.inin}

Let $g \in \mathbb{F}$ and let $\mathsf{C},\mathsf{D} \in \mathrm{NSPD}_g$. Observe that for any $\zeta \in \mathbb{D}$, the space $\mathscr{X}(\mathsf{C}^\zeta)$ has a canonical basis $\{\Theta_\mathsf{C}(h):h \in \mathcal{K}_g\}$. We will adopt the convention that for $\lambda \in \mathbb{D}$ a vector $x \in \mathscr{X}(\mathsf{C}^\zeta)$ is identified with the vector in $\mathscr{X}(\mathsf{C}^\lambda)$ having the same coordinates with respect to the canonical basis.\\
\\
Let $\zeta,\mu \in \mathbb{D}$ and consider the extensions $\mathsf{C}^\zeta$ and $\mathsf{D}^\mu$. We consider an additive perturbation $\chi \varsigma$ to the parameters $\zeta$ and $\mu$ where $\varsigma \in \partial \mathbb{D}$ and $\chi \in \mathbb{R}$ is sufficiently small that $\max(|\zeta+\chi \varsigma|,|\mu+\chi \varsigma|) < 1$. It will be convenient to introduce the asymptotic notations $O(\cdot)$ and $o(\cdot)$ with respect to the limit $\chi \to 0$.\\
\\
Let $p$ be the orthogonal projection from $\mathscr{X}(\mathsf{C})_\bullet$ onto $\mathrm{core}(\mathscr{X}(\mathsf{C})_\bullet)$ and let $q$ be the orthogonal projection from $\mathscr{X}(\mathsf{D})_\bullet$ onto $\mathrm{core}(\mathscr{X}(\mathsf{D})_\bullet)$. We introduce the following notations. \begin{align} S & = \frac{(I-p)  \Theta_{\mathsf{C}}(g)}{||(I-p)  \Theta_{\mathsf{C}}(g)||} \label{eq.class-1} \\ S' &= \frac{(I-p)  \Theta_{\mathsf{C}}(e)}{||(I-p)  \Theta_{\mathsf{C}}(e)||} \label{eq.class-2} \\ T & = \frac{(I-q)  \Theta_{\mathsf{D}}(g)}{||(I-p)  \Theta_{\mathsf{D}}(g)||} \label{eq.class-3} \\ T'& = \frac{(I-q)  \Theta_{\mathsf{D}}(e)}{||(I-q)  \Theta_{\mathsf{D}}(e)||} \label{eq.class-4} \end{align}

\subsection{Energy increases require extension components}

\begin{proposition} Let $x \in \mathscr{X}(\mathsf{C})_\bullet$ and write $x = \alpha S+\alpha'S'+x'$ for $x' \in \mathrm{core}(\mathscr{X}(\mathsf{C})_\bullet)$. If $||t[\mathsf{C}^\zeta,\mathsf{D}^\mu]x||^2 > \mathfrak{e}(\mathsf{C},\mathsf{D})||x||^2$ then both $\alpha$ and $\alpha'$ are nonzero. \label{prop.nonzero} \end{proposition}

\begin{proof}[Proof of Proposition \ref{prop.nonzero}] This is immediate from the observation that if one of $\alpha$ and $\alpha'$ is zero then $x$ lies in one of the distinguished subspaces of the partial Hilbert space $\mathscr{X}(\mathsf{C})_\bullet$. \end{proof}

\subsection{Energy increases give one dimensional norm achievers} \label{seg.onedim}

\begin{proposition} Suppose that $\zeta,\mu \in \mathbb{D}$ are such that \begin{equation} \label{eq.class-5} \mathfrak{e}(\mathsf{C}^\zeta,\mathsf{D}^\mu) > \mathfrak{e}(\mathsf{C},\mathsf{D}) \end{equation} Then the space of vectors which achieve the norm of $t[\mathsf{C}^\zeta,\mathsf{D}^\mu]$ is one-dimensional. \label{prop.carol} \end{proposition}

\begin{proof}[Proof of Proposition \ref{prop.carol}] Suppose toward a contradiction that $x$ and $y$ are orthogonal unit vectors in $\mathscr{X}(\mathsf{C}^\zeta)$ which achieve the norm of $t[\mathsf{C}^\zeta,\mathsf{D}^\mu]$. Write $x = \alpha S + \alpha' S' + x'$ for $x' \in \mathrm{core}(\mathscr{X}(\mathsf{C})_\bullet)$ and write $y = \beta S + \alpha'S' + y'$ for $y' \in \mathrm{core}(\mathscr{X}(\mathsf{C})_\bullet)$. Proposition \ref{prop.nonzero} and the hypothesis (\ref{eq.class-5}) implies that $\alpha \neq 0$. Consider the vector $z = \beta \alpha^{-1}x - y$. This vector $z$ also achieves the norm of $t[\mathsf{C},\mathsf{D}]$. Since $z$ has no $S$ component using Proposition \ref{prop.nonzero} we obtain a contradiction to the hypothesis (\ref{eq.class-5}). \end{proof}

\subsection{Initial bounds on coefficients and energies}

\begin{proposition} Let $x \in \mathscr{X}(\mathsf{C}^\zeta)$ and write $x = \alpha S + \alpha' S' + x'$ for $x' \in \mathrm{core}(\mathscr{X}(\mathsf{C})_\bullet)$. Then we have $|\alpha||\alpha'| \leq (1-|\zeta|)^{-1}||x||^2$. \label{prop.anita+1} \end{proposition}

\begin{proof}[Proof of Proposition \ref{prop.anita+1}]  We have \begin{align} ||x||^2 &= ||\alpha S + \alpha'S' + x'||^2 \label{eq.ligro-1} \\ & \geq ||\alpha S + \alpha'S'||^2 \label{eq.ligro-2} \\ & = |\alpha|^2 + |\alpha'|^2 + 2\, \mathrm{Re}(\zeta \alpha \ov{\alpha'}) \label{eq.ligro-3} \\ & \geq |\alpha|^2 + |\alpha'|^2 - 2|\alpha||\alpha'| |\zeta| \nonumber \\ & = (1-|\zeta|)(|\alpha|^2 + |\alpha'|^2) + |\zeta|(|\alpha|^2 + |\alpha'|^2- 2|\alpha||\alpha'|) \label{eq.ligro-5} \\ & \geq (1-|\zeta|)(|\alpha|^2+|\alpha'|^2) \label{eq.ligro-6} \end{align}

Here, \begin{itemize} \item (\ref{eq.ligro-2}) follows from (\ref{eq.ligro-1}) since $\alpha S + \alpha'S' \perp x'$, \item (\ref{eq.ligro-3}) follows from (\ref{eq.ligro-2}) since by construction we have $\langle S, S' \rangle = \zeta$, \item and (\ref{eq.ligro-6}) follows from (\ref{eq.ligro-5}) since \[ |\alpha|^2+|\alpha'|^2-2|\alpha||\alpha'| = (|\alpha|-|\alpha'|)^2 \geq 0 \] \end{itemize}Therefore we have $\max(|\alpha|^2,|\alpha'|^2) \leq (1-|\zeta|)^{-1}$ and so the proof of Proposition \ref{prop.anita+1} is complete. \end{proof}

\begin{proposition} We have $\mathfrak{e}(\mathsf{C}^{\zeta+\chi \varsigma},\mathsf{D}^\mu) = \mathfrak{e}(\mathsf{C}^\zeta,\mathsf{D}^\mu) + O(\chi)$ and $\mathfrak{e}(\mathsf{C}^\zeta,\mathsf{D}^{\mu+\chi \varsigma}) = \mathfrak{e}(\mathsf{C}^\zeta,\mathsf{D}^\mu) + O(\chi)$. \label{prop.anita-0} \end{proposition}

\begin{proof}[Proof of Proposition \ref{prop.anita-0}] Let $x_\chi \in \mathscr{X}(\mathsf{C}^{\zeta+\chi \varsigma})$ be a unit vector which achieves the norm of $t[\mathsf{C}^{\zeta+\chi \varsigma},\mathsf{D}^\mu]$. Write $||\cdot||_\chi$ for the norm on $\mathscr{X}(\mathsf{C}^{\zeta+\chi \varsigma})$ and write $x_\chi = \alpha_\chi S + \alpha_\chi' S' + x'_\chi$ for $x'_\chi \in \mathrm{core}(\mathscr{X}(\mathsf{C})_\bullet)$. We have $(1-|\zeta+\chi \varsigma|)^{-1} = O(1)$, so that Proposition \ref{prop.anita+1} implies $|\alpha_\chi||\alpha'_\chi| = O(1)$. We have \begin{align} ||x_\chi||^2_0 & = ||x_\chi||^2_\chi - 2\chi\,\mathrm{Re}(\varsigma \alpha_\chi \ov{\alpha'_\chi}) \label{eq.done-26} \\ & = 1 - 2\chi\,\mathrm{Re}(\varsigma \alpha_\chi \ov{\alpha_\chi'}) \label{eq.done-27} \\ & \leq 1+ 2 \chi|\alpha_\chi||\alpha_\chi'| \nonumber \\ & \leq 1+O(\chi) \label{eq.done-25} \end{align} Here, (\ref{eq.done-27}) follows from (\ref{eq.done-26}) since we assumed that $x_\chi$ was a unit vector in $\mathscr{X}(\mathsf{C}^{\zeta+\chi \varsigma})$. Therefore we have \begin{align} \mathfrak{e}(\mathsf{C}^\zeta,\mathsf{D}^\mu) & \geq ||t[\mathsf{C}^\zeta,\mathsf{D}^\mu] x_\chi||^2\,||x_\chi||_0^{-2} \label{eq.starb-50} \\ & = ||t[\mathsf{C}^{\zeta+\chi \varsigma},\mathsf{D}^\mu] x_\chi||^2\,||x_\chi||_0^{-2} \label{eq.starb-51} \\ & = \mathfrak{e}(\mathsf{C}^{\zeta+\chi \varsigma},\mathsf{D}^\mu)||x_\chi||_0^{-2} \label{eq.starb-52} \\ & \geq \mathfrak{e}(\mathsf{C}^{\zeta+\chi \varsigma},\mathsf{D}^\mu)(1+O(\chi))^{-1} \label{eq.starb-53} \\ & = \mathfrak{e}(\mathsf{C}^{\zeta+\chi \varsigma},\mathsf{D}^\mu)(1-O(\chi)) \label{eq.starb-54} \\ & = \mathfrak{e}(\mathsf{C}^{\zeta+\chi\varsigma},\mathsf{D}^\mu)-O(\chi) \label{eq.starb-55} \end{align} 

Here, \begin{itemize} \item (\ref{eq.starb-51}) follows from (\ref{eq.starb-50}) since $||t[\mathsf{C}^{\zeta+\chi \varsigma},\mathsf{D}^\mu] x_\chi||$ is computed in $\mathscr{X}(\mathsf{D}^\mu)$ and hence does not depend on $\chi$, \item (\ref{eq.starb-52}) follows from (\ref{eq.starb-51}) since we assumed $x_\chi$ is a unit vector achieving the norm of $t[\mathsf{C}^{\zeta+\chi \varsigma},\mathsf{D}^\mu]$, \item and (\ref{eq.starb-53}) follows from (\ref{eq.starb-52}) by (\ref{eq.done-25}) \item (\ref{eq.starb-55}) follows from (\ref{eq.starb-54}) since Proposition \ref{prop.anita+1} implies $\mathfrak{e}(\mathsf{C}^{\zeta+\chi \varsigma},\mathsf{D}^\mu) = O(1)$. \end{itemize}

Now, let $x_\chi \in \mathscr{X}(\mathsf{C}^\zeta)$ be a unit vector which achieves the norm of $t[\mathsf{C}^\zeta,\mathsf{D}^{\mu+\chi \varsigma}]$. We modify the notation $||\cdot||_\chi$ to now refer to the norm of $\mathscr{X}(\mathsf{D}^{\mu+\chi \varsigma})$ and write $t[\mathsf{C}^\zeta,\mathsf{D}^{\mu+\chi \varsigma}] x_\chi = \beta_\chi T+ \beta_\chi ' T' + x'_\chi$ for $x' \in \mathrm{core}(\mathscr{X}(\mathsf{D})_\bullet)$. We have

\begin{align} \mathfrak{e}(\mathsf{C}^\zeta,\mathsf{D}^{\mu+\chi \varsigma}) &= ||t[\mathsf{C}^\zeta,\mathsf{D}^{\mu+\chi \varsigma}] x_\chi||_\chi^2 \nonumber \\ & = ||t[\mathsf{C}^\zeta,\mathsf{D}^{\mu+\chi \varsigma}] x_\chi||_0^2 + 2\chi\, \mathrm{Re}(\varsigma \beta_\chi \ov{\beta'_\chi}) \label{eq.rumba-1} \\ & \leq \mathfrak{e}(\mathsf{C}^\zeta,\mathsf{D}^\mu) + 2 \chi \, \mathrm{Re}(\varsigma \beta_\chi \ov{\beta'_\chi}) + O(\chi) \label{eq.rumba-2} \\ & \leq \mathfrak{e}(\mathsf{C}^\zeta,\mathsf{D}^\mu) + O(\chi)  \label{eq.rumba-3} \end{align}

Here, (\ref{eq.rumba-2}) follows from (\ref{eq.rumba-1}) since $x_\chi$ is a unit vector in $\mathscr{X}(\mathsf{C}^\zeta)$ and (\ref{eq.rumba-3}) follows from (\ref{eq.rumba-2}) since $||t[\mathsf{C}^\zeta,\mathsf{D}^{\mu + \varsigma \chi}]x_\chi|| = O(1)$ and therefore Proposition \ref{prop.anita+1} shows $\max(|\beta_\chi|,|\beta'_\chi|) = O(1)$. \end{proof}

\subsection{Differentiability of the energy}

\begin{proposition} Suppose the space of vectors which achieve the norm of $t[\mathsf{C},\mathsf{D}]$ is one dimensional and let $\varsigma \in \partial \mathbb{D}$. Then for sufficiently small $\chi$ the quantity $\mathfrak{e}(\mathsf{C}^{\zeta+\chi \varsigma},\mathsf{D}^\mu)$ is a differentiable function of $\chi$. Moreover, for every such $\chi$ there is a vector $x_\chi \in \mathscr{X}(\mathsf{C}^{\zeta+\chi \varsigma})$ which achieves the norm of $t[\mathsf{C}^{\zeta+\chi \varsigma},\mathsf{D}^\mu]$ and such that the coordinates of $x_\chi$ with respect to the canonical basis are differentiable functions of $\chi$. \\
\\
Similarly, for sufficiently small $\chi$ the quantity $\mathfrak{e}(\mathsf{C}^\zeta,\mathsf{D}^{\mu+\chi \varsigma})$ is a differentiable function of $\chi$. Moreover, for every such $\chi$ there is a vector $y_\chi \in \mathscr{X}(\mathsf{D}^{\mu+\chi \varsigma})$ which is the image of a vector achieving the norm of $t[\mathsf{C}^\zeta,\mathsf{D}^{\mu+\chi \varsigma}]$ and such that the coordinates of $y_\chi$ with respect to the canonical basis are differentiable functions of $\chi$. We may assume that $y_0$ is the image of a unit vector. \label{prop.anita-1}  \end{proposition}

\begin{proof}[Proof of Proposition \ref{prop.anita-1}] We establish Proposition \ref{prop.anita-1} for perturbations of $\zeta$. The case of perturbations of $\mu$ can be established using a similar method. Let $B_\chi$ be the orthonormal basis for $\mathscr{X}(\mathsf{C}^{\zeta+\chi \varsigma})$ obtained by applying the Gram-Schmidt orthonormalization procedure to the basis \begin{equation} \label{eq.starb-4} \{\Theta_{\mathsf{C}^{\zeta+\chi \varsigma}}(h):h \in \mathcal{K}_g \setminus \{g\}\} \cup \{\ov{\varsigma} \Theta_{\mathsf{C}^{\zeta+\chi \varsigma}}(g) \} \end{equation} Let $\mathcal{N}_\chi$ be the matrix which changes coordinates from the basis in (\ref{eq.starb-4}) to $B_\chi$. Since we have introduced the phase $\ov{\varsigma}$ to the last vector in (\ref{eq.starb-4}), the matrix $\mathcal{N}_\chi$ is a real perturbation of the matrix $\mathcal{N}_0$.\ \\
\\
Also let $D$ be the orthonormal basis for $\mathscr{X}(\mathsf{D}^\mu)$ obtained by applying the Gram-Schmidt orthonormalization procedure to the basis $\{\Theta_{\mathsf{D}^\mu}(h): h \in \mathcal{K}_g \}$ and let $\mathcal{O}$ be the matrix which changes coordinates from the basis in (\ref{eq.starb-4}) to $D$. With respect to the bases $B_\chi$ and $D$ the matrix of the transport operator $t[\mathsf{C}^{\zeta + \chi \varsigma},\mathsf{D}^\mu]$ is given by $\mathcal{O} \mathcal{N}_\chi^{-1}$. The advantage of writing the matrix with respect to these orthonormal bases is that the matrix of the adjoint $t[\mathsf{C}^{\zeta+\chi \varsigma},\mathsf{D}^\mu]^\ast$ is given by the conjugate transpose of the matrix $\mathcal{O} \mathcal{N}_\chi^{-1}$, which we denote simply by $(\mathcal{O} \mathcal{N}_\chi^{-1})^\ast$. Let $\xi_\chi = (\mathcal{O} \mathcal{N}_\chi^{-1})^\ast \mathcal{O} \mathcal{N}_\chi^{-1}$\\
\\
The second clause in \ref{prop.star} implies the entries of $\xi_\chi$ are differentiable functions of $\chi$. Therefore Theorem 6.1 in \cite{MR1335452} implies that there differentiable functions $\kappa_1(\chi),\ldots,\kappa_n(\chi)$ which represent the eigenvalues of $\xi_\chi$ for sufficiently small values of $\chi$. We may assume that $\kappa_m(0) \geq \kappa_{m+1}(0)$ for all $m \in [n-1]$. Since $\kappa_1(0) > \kappa_2(0)$, for all sufficiently small $\chi$ the function $\kappa_1(\chi)$ represents the norm of $\xi_\chi$. This complete the proof of the first claim in Proposition \ref{prop.anita-1}. \\
\\
Recall that the index set for the matrix $\xi_\chi$ is $\mathcal{K}_g$. Theorem 6.1 in \cite{MR1335452} also implies that for each $\chi$ there are vectors $\varphi_1(\chi),\ldots,\varphi_n(\chi) \in \ell^2(\mathcal{K}_g)$ such that $\varphi_m(\chi)$ is an eigenvector of $\xi_\chi$ with eigenvalue $\kappa_m(\chi)$ and such that the coordinates of $\varphi_m(\chi)$ are differentiable functions of $\chi$. These numerical vectors represent the coordinates of the singular vectors for $t[\mathsf{C}^{\zeta+\chi \varsigma},\mathsf{D}^\mu]$ in the basis $B_\chi$. In order to change the coordinates back to the basis in (\ref{eq.starb-4}) we need to multiply $\varphi_m(\chi)$ by $\mathcal{N}_\chi^{-1}$. Since the entries of $\mathcal{N}_\chi^{-1}$ are differentiable functions of $\chi$ the second claim in Proposition \ref{prop.anita-1} follows. \end{proof}

\subsection{Calculation of derivatives}

\begin{proposition} \label{prop.anita-2} Suppose that for all sufficiently small $\chi$ the space of vectors which achieve the norm of $t[\mathsf{C}^{\zeta+\varsigma \chi},\mathsf{D}^\mu]$ is one-dimensional. Let $x \in \mathscr{X}(\mathsf{C}^\zeta)$ be a unit vector which achieves the norm of $t[\mathsf{C}^{\zeta},\mathsf{D}^\mu]$ and write $x = \alpha S + \alpha'S' + x'$ for $x' \in \mathrm{core}(\mathscr{X}(\mathsf{C})_\bullet)$. Then we have \[ \frac{\dee}{\dee \chi} \mathfrak{e}(\mathsf{C}^{\zeta+\varsigma \chi},\mathsf{D}^\mu) \bigg \vert_{\chi = 0} = -2\, \mathfrak{e}(\mathsf{C}^\zeta,\mathsf{D}^\mu) \,\mathrm{Re}(\varsigma \alpha \ov{\alpha'}) \] Also let $y \in \mathscr{X}(\mathsf{D}^\mu)$ be given by $y = t[\mathsf{C}^\zeta,\mathsf{D}^\mu] x$ and write $y = \beta T + \beta' T' + y'$ for $y' \in \mathscr{X}(\mathsf{D})$. Then we have \[ \frac{\dee}{\dee \chi} \mathfrak{e}(\mathsf{C}^\zeta,\mathsf{D}^{\mu + \varsigma \chi}) \bigg \vert_{\chi = 0} = 2\, \mathrm{Re}(\varsigma \beta \ov{\beta'} ) \] \end{proposition}

\begin{proof}[Proof of Proposition \ref{prop.anita-2}] Let $x_\chi \in \mathscr{X}(\mathsf{C}^{\zeta+\chi \varsigma})$ be as in Proposition \ref{prop.anita-1}. Write $x_\chi = \alpha_\chi S + \alpha'_\chi S' + x'_\chi$ for $x'_\chi \in \mathrm{core}(\mathscr{X}(\mathsf{C})_\bullet)$.\\
\\
In order to distinguish between the norms on different spaces, we will write $||\cdot||_\chi$ for the norm on $\mathscr{X}(\mathsf{C}^{\zeta+\chi \varsigma})$. We have \begin{align} ||x_\chi||^2_\chi & = ||x_\chi||^2_0 + 2\chi\, \mathrm{Re}(\varsigma \alpha_\chi \ov{\alpha'_\chi}) \nonumber \\ & = ||x_\chi||^2_0 + 2\chi\, \mathrm{Re}(\varsigma(\alpha_\chi-\alpha)\ov{\alpha'_\chi}) \nonumber  \\ & \quad \quad + 2\chi\, \mathrm{Re}(\varsigma \alpha(\ov{\alpha'_\chi} - \ov{\alpha'})) + 2\chi \, \mathrm{Re}(\varsigma \alpha \ov{\alpha'}) \label{eq.starb-10} \end{align}

By Proposition \ref{prop.anita-1} we see that $\alpha_\chi = \alpha + o(1)$ and $\alpha'_\chi = \alpha' + o(1)$. Therefore in (\ref{eq.starb-10}) we have that the terms \[ 2\chi\, \mathrm{Re}(\varsigma(\alpha_\chi-\alpha)\ov{\alpha'_\chi}) \] and \[ 2\chi\, \mathrm{Re}(\varsigma \alpha(\ov{\alpha'_\chi} - \ov{\alpha'})) \] are $o(\chi)$. It follows that \begin{equation} \label{eq.starb-11} ||x_\chi||^2_\chi = ||x_\chi||^2_0 + 2 \chi \mathrm{Re}(\varsigma \alpha \ov{\alpha'}) + o(\chi) \end{equation}

We compute 

\begin{align} \mathfrak{e}(\mathsf{C}^{\zeta + \chi \varsigma},\mathsf{D}^{\mu}) & = ||t[\mathsf{C}^{\zeta + \chi \varsigma},\mathsf{D}^\mu] x_{\chi}||^2 \, ||x_{\chi}||_{\chi}^{-2} \label{eq.yon-0.5} \\ & = ||t[\mathsf{C}^{\zeta + \chi \varsigma},\mathsf{D}^\mu] x_{\chi}||^2 (||x_{\chi}||^2_0 + 2\chi \, \mathrm{Re}(\varsigma \alpha \ov{\alpha'}) + o(\chi) )^{-1} \label{eq.yon-1} \\ & = ||t[\mathsf{C}^{\zeta + \chi \varsigma},\mathsf{D}^\mu] x_{\chi}||^2\,||x_\chi||_0^{-2} \nonumber \\ & \quad \quad \cdot (1 + 2\chi||x_\chi||^{-2}_0 \, \mathrm{Re}(\varsigma \alpha \ov{\alpha'}) + o(\chi) )^{-1} \label{eq.yon-2} \\ & = ||t[\mathsf{C}^{\zeta + \chi \varsigma},\mathsf{D}^\mu] x_{\chi}||^2\,||x_\chi||_0^{-2} \nonumber\\ & \quad \quad \cdot  (1 - 2\chi||x_\chi||^{-2}_0 \, \mathrm{Re}(\varsigma \alpha \ov{\alpha'}) + o(\chi)) \label{eq.yon-3}  \\ & = ||t[\mathsf{C}^{\zeta + \chi \varsigma},\mathsf{D}^\mu] x_{\chi}||^2\,||x_\chi||_0^{-2} (1 - 2\chi \, \mathrm{Re}(\varsigma \alpha \ov{\alpha'}) + o(\chi)) \label{eq.yon-4}   \\ & = ||t[\mathsf{C}^{\zeta},\mathsf{D}^\mu] x_{\chi}||^2\,||x_\chi||_0^{-2} (1 - 2\chi \, \mathrm{Re}(\varsigma \alpha \ov{\alpha'}) + o(\chi)) \label{eq.yon-5}   \\ & \leq \mathfrak{e}(\mathsf{C}^\zeta,\mathsf{D}^\mu)(1-2\chi \, \mathrm{Re}(\varsigma \alpha \ov{\alpha'}) + o(\chi)) \label{eq.yon-6}  \end{align} 
 
This computation can be justified as follows. \begin{itemize} \item (\ref{eq.yon-1}) follows from (\ref{eq.yon-0.5}) by (\ref{eq.starb-11}) \item (\ref{eq.yon-2}) follows from (\ref{eq.yon-1}) since Proposition \ref{prop.anita-1} together with the hypothesis that $||x||_0=1$ implies $||x_\chi||_0 = 1+o(1)$, so that $||x_\chi||_0$ can be absorbed into the $o(\chi)$ term.  \item (\ref{eq.yon-3}) follows from (\ref{eq.yon-2}) since the higher terms in the geometric series can be absorbed into the $o(\chi)$ term \item (\ref{eq.yon-4}) follows from (\ref{eq.yon-3}) again since $||x_\chi||_0 = 1+o(1)$. \item (\ref{eq.yon-5}) follows from (\ref{eq.yon-4}) since $||t[\mathsf{C}^{\zeta + \chi \varsigma},\mathsf{D}^\mu] x_\chi||$ is computed in $\mathscr{X}(\mathsf{D}^\mu)$. \end{itemize}

From the above computation we see that \begin{equation} \label{eq.starb-12} \frac{\dee}{\dee \chi} \mathfrak{e}(\mathsf{C}^{\zeta+\chi \varsigma}) \bigg \vert_{\chi = 0} \leq -2\,\mathfrak{e}(\mathsf{C}^\zeta,\mathsf{D}^\mu)\,\mathrm{Re}(\varsigma \alpha \ov{\alpha'}) \end{equation}

On the other hand, we have \begin{align} \mathfrak{e}(\mathsf{C}^{\zeta+\chi \varsigma},\mathsf{D}^\mu) & \geq ||t[\mathsf{C}^{\zeta + \chi \varsigma},\mathsf{D}^\mu] x||^2 \, ||x||_{\chi}^{-2} \label{eq.starb-13} \\ & = ||t[\mathsf{C}^{\zeta},\mathsf{D}^\mu] x||^2 \, ||x||_{\chi}^{-2} \label{eq.starb-14} \\ & = \mathfrak{e}(\mathsf{C}^\zeta,\mathsf{D}^\mu)||x||_\chi^{-2} \label{eq.starb-15} \\ & = \mathfrak{e}(\mathsf{C}^\zeta,\mathsf{D}^\mu)(1+2\chi\,\mathrm{Re}(\varsigma \alpha \ov{\alpha'}))^{-1} \label{eq.starb-15.5}  \\ & = \mathfrak{e}(\mathsf{C}^\zeta,\mathsf{D}^\mu)(1-2\chi\,\mathrm{Re}(\varsigma \alpha \ov{\alpha'})+o(\chi)) \label{eq.starb-16} \end{align}

The above computation can be justified as follows. \begin{itemize} \item (\ref{eq.starb-14}) follows from (\ref{eq.starb-13}) since $||t[\mathsf{C}^{\zeta + \chi \varsigma},\mathsf{D}^\mu] x||$ is computed in $\mathscr{X}(\mathsf{D}^\mu)$ and hence is independent of $\chi$. \item (\ref{eq.starb-15}) follows from (\ref{eq.starb-14}) since we assumed $x$ was a unit vector in $\mathscr{X}(\mathsf{C}^\zeta)$ which achieves the norm of $t[\mathsf{C}^\zeta,\mathsf{D}^\mu]$. \item (\ref{eq.starb-16}) follows from (\ref{eq.starb-15.5}) since the higher terms in the geometric series can be absorbed into the $o(\chi)$ term. \end{itemize}

From the above computation and (\ref{eq.starb-12}) we obtain \[ \frac{\dee}{\dee \chi} \mathfrak{e}(\mathsf{C}^{\zeta+\chi \varsigma}) \bigg \vert_{\chi = 0} = -2\,\mathfrak{e}(\mathsf{C}^\zeta,\mathsf{D}^\mu)\,\mathrm{Re}(\varsigma \alpha \ov{\alpha'}) \] This establishes Proposition \ref{prop.anita-2} for perturbations of $\zeta$.\\
\\
Now, let $y_\chi \in \mathscr{X}(\mathsf{D}^{\mu+\chi \varsigma})$ be as in Proposition \ref{prop.anita-1}, so that $y_\chi$ is the image of a vector which achieves the norm of $t[\mathsf{C}^\zeta,\mathsf{D}^{\mu+\chi \varsigma}]$ and the coordinates of $y_\chi$ with respect to the canonical basis (\ref{eq.lewis-1}) are continuous functions of $\chi$. Write $y_\chi = \beta_\chi T + \beta'_\chi T' + y'_\chi$ for $y' \in \mathrm{core}(\mathscr{X}(\mathsf{D})_\bullet)$. We will alter the notation $||\cdot||_\chi$ to now refer to the norm on $\mathsf{D}^{\mu+\chi \varsigma}$. We have \[ ||y_\chi||_\chi  = ||y_\chi||_0 + 2\chi\,\mathrm{Re}(\varsigma \beta_\chi \ov{\beta'_\chi}) \] so that the same argument used to establish (\ref{eq.starb-11}) shows that \begin{equation} ||y_\chi||_\chi = ||y_\chi||_0 + 2 \chi\, \mathrm{Re}(\varsigma \beta \ov{\beta'}) + o(\chi) \label{eq.starb-19} \end{equation} We compute \begin{align} \mathfrak{e}(\mathsf{C}^\zeta,\mathsf{D}^{\mu+\chi \varsigma}) & = ||y_\chi||^2_\chi ||t[\mathsf{D}^{\mu+\chi \varsigma},\mathsf{C}]y_\chi||^{-2} \label{eq.starb-20} \\  & = ||y_\chi||^2_\chi ||t[\mathsf{D}^\mu,\mathsf{C}^\zeta] y_\chi||^{-2} \label{eq.starb-21} \\  &= (||y_\chi||^2_0 + 2\chi\,\mathrm{Re}(\varsigma \beta \ov{\beta}) + o(\chi))||t[\mathsf{D}^\mu,\mathsf{C}^\zeta] y_\chi||^{-2} \label{eq.starb-22} \\ & = ||y_\chi||^2_0 ||t[\mathsf{D}^\mu,\mathsf{C}^\zeta] y_\chi||^{-2}  \nonumber \\ & \quad \quad + 2 \chi \, \mathrm{Re}(\varsigma \beta \ov{\beta})||t[\mathsf{D}^\mu,\mathsf{C}^\zeta] y_\chi||^{-2} + o(\chi) \label{eq.pasta-1} \\ & \leq \mathfrak{e}(\mathsf{C}^\zeta,\mathsf{D}^\mu) + 2 \chi \, \mathrm{Re}(\varsigma \beta \ov{\beta})||t[\mathsf{D}^\mu,\mathsf{C}^\zeta] y_\chi||^{-2} + o(\chi) \label{eq.starb-24} \\ & = \mathfrak{e}(\mathsf{C}^\zeta,\mathsf{D}^\mu) + 2 \chi \, \mathrm{Re}(\varsigma \beta \ov{\beta})(1+o(1)) + o(\chi) \label{eq.starb-25}  \\ & = \mathfrak{e}(\mathsf{C}^\zeta,\mathsf{D}^\mu) + 2 \chi \, \mathrm{Re}(\varsigma \beta \ov{\beta}) + o(\chi) \nonumber \end{align} This computation can be justified as follows. \begin{itemize} \item (\ref{eq.starb-21}) follows from (\ref{eq.starb-20}) since $||t[\mathsf{D}^\mu,\mathsf{C}^\zeta] y_\chi||$ is computed in $\mathscr{X}(\mathsf{C}^\zeta)$. \item (\ref{eq.starb-22}) follows from (\ref{eq.starb-21}) by (\ref{eq.starb-19}) \item (\ref{eq.pasta-1}) follows from (\ref{eq.starb-22}) since $y_0$ was assumed to be the image of a unit vector in $\mathscr{X}(\mathsf{C}^\zeta)$ and therefore Proposition \ref{prop.anita-1} implies \[ ||t[\mathsf{D}^\mu,\mathsf{C}^\zeta] y_\chi||^{-2} = 1+o(1) \] \item (\ref{eq.starb-25}) follows from (\ref{eq.starb-24}) by the same reasoning as in the previous bullet. \end{itemize} 

The above computation shows that \begin{equation} \frac{\dee}{\dee \chi} \mathfrak{e}(\mathsf{C}^\zeta,\mathsf{D}^{\mu+\chi \varsigma}) \bigg \vert_{\chi = 0} \leq 2\, \mathrm{Re}(\varsigma \beta \ov{\beta}) \label{eq.starb-27} \end{equation}

On the other hand, we have \begin{align} \mathfrak{e}(\mathsf{C}^\zeta,\mathsf{D}^{\mu+\chi \varsigma}) & \geq ||y||^2_\chi ||t[\mathsf{D}^{\mu+\varsigma \chi},\mathsf{C}^\zeta]y||^{-2} \label{eq.starb-28} \\ & = ||y||_\chi^2 \label{eq.starb-29} \\ & = ||y||^2_0 + 2\chi \, \mathrm{Re}(\varsigma \beta \ov{\beta'}) \label{eq.starb-30} \\ & = \mathfrak{e}(\mathsf{C}^\zeta,\mathsf{D}^\mu) + 2 \chi\, \mathrm{Re}(\varsigma \beta \ov{\beta'}) \label{eq.starb-31} \end{align} Here, (\ref{eq.starb-29}) follows from (\ref{eq.starb-28}) since $y$ was assumed to be the image of a unit vector in $\mathscr{X}(\mathsf{C}^\zeta)$ and (\ref{eq.starb-31}) follows from (\ref{eq.starb-30}) since $y$ was assumed to achieve the norm of $t[\mathsf{C}^\zeta,\mathsf{D}^\mu]$. \end{proof}

\subsection{Relationship between scalars before and after transport} \label{seg.before}

\begin{proposition} \label{prop.anita-3} Let $x \in \mathscr{X}(\mathsf{C})_\bullet$ and write $x = \alpha S + \alpha' S' + x'$ for $x' \in \mathrm{core}(\mathscr{X}(\mathsf{C})_\bullet)$. Let $\zeta,\mu \in \mathbb{D}$ be arbitrary and write $t[\mathsf{C}^\zeta,\mathsf{D}^\mu] x = \beta T + \beta' T' + y'$ for $y' \in \mathscr{X}(\mathsf{D})$. Then we have \begin{equation} \beta \ov{\beta'} = \alpha \ov{\alpha'} \frac{||(I-p)  \Theta_{\mathsf{C}}(g)||\,||(I-p)  \Theta_{\mathsf{C}}(e)||}{||(I-p)  \Theta_{\mathsf{D}}(g)||\,||(I-q)  \Theta_{\mathsf{D}}(e)||} \label{eq.sunny-1} \end{equation} \end{proposition}

\begin{proof}[Proof of Proposition \ref{prop.anita-3}] Inspecting the definition (\ref{eq.class-1}) we see that the quantity $\alpha  ||(I-p)  \Theta_{\mathsf{C}}(g)||$ is the coefficient of $\Theta_\mathsf{C}(g)$ in the expression of $x$ with respect to the canonical basis (\ref{eq.lewis-1}) for $\mathscr{X}(\mathsf{C})_\bullet$. From Definition \ref{def.transport} we see that $\alpha  ||(I-p)  \Theta_{\mathsf{C}}(g)||$ is the coefficient of $\Theta_\mathsf{D}(g)$ in the expression of $t[\mathsf{C}^\zeta,\mathsf{D}^\mu] x$ with respect to the canonical basis (\ref{eq.lewis-1}) for $\mathscr{X}(\mathsf{D})_\bullet$. Thus from (\ref{eq.class-3}) we see that \[ \beta = \alpha\frac{ ||(I-p)  \Theta_{\mathsf{C}}(g)||}{||(I-p)  \Theta_{\mathsf{D}}(g)||} \] Similarly, we have \[ \beta' = \alpha' \frac{ ||(I-p)  \Theta_{\mathsf{C}}(e)||}{||(I-p)  \Theta_{\mathsf{D}}(e)||} \] and Proposition \ref{prop.anita-3} follows. \end{proof}

\subsection{Completing a single step} \label{oo}

\begin{proposition} \label{state} Let $g \in \mathbb{F}$ and $\mathsf{C}$ and $\mathsf{D}$ be elements of $\mathrm{NSPD}_g$ such that the pair $(\mathsf{C},\mathsf{D})$ has singular degeneracies. Then for any $\mu \in \mathbb{D}$ there exists $\zeta \in \mathbb{D}$ such that $\mathfrak{e}(\mathsf{C}^\zeta,\mathsf{D}^\mu) = \mathfrak{e}(\mathsf{C},\mathsf{D})$. \end{proposition}

\begin{proof}[Proof of Proposition \ref{state}] Let $g,\mathsf{C}$ and $\mathsf{D}$ be as in Proposition \ref{state}. Fix $\mu \in \mathbb{D}$. The hypothesis of completely singular degeneracies implies that the function $\zeta \mapsto \mathfrak{e}(\mathsf{C}^\zeta,\mathsf{D}^\mu)$ attains its minimum on $\mathbb{D}$. Fix $\zeta$ at which this minimum is attained and suppose toward a contradiction that $\mathfrak{e}(\mathsf{C}^\zeta,\mathsf{D}^\mu) > \mathfrak{e}(\mathsf{C},\mathsf{D})$. Let $x \in \mathscr{X}(\mathsf{C}^\zeta)$ be a unit vector which achieves the norm of $t[\mathsf{C}^{\zeta},\mathsf{D}^\mu]$ and write $x = \alpha S + \alpha'S' + x'$ for $x' \in \mathrm{core}(\mathscr{X}(\mathsf{C})_\bullet)$. Since we must have \[ \frac{\dee}{\dee \chi} \mathfrak{e}(\mathsf{C}^{\zeta+\chi \varsigma},\mathsf{D}^\mu) \bigg \vert_{\chi = 0} = 0 \] for all $\varsigma \in \partial \mathbb{D}$, the first clause in Proposition \ref{prop.anita-2} implies $\mathrm{Re}(\varsigma \alpha\ov{\alpha'}) = 0$ for all $\varsigma \in \partial \mathbb{D}$. Therefore $\alpha\ov{\alpha'} = 0$. Using Proposition \ref{prop.nonzero} this contradicts the hypothesis that $||t[\mathsf{C}^\zeta,\mathsf{D}^\mu] x||^2 > \mathfrak{e}(\mathsf{C},\mathsf{D})$. \end{proof}

\subsection{Recursive extension procedure} \label{ooo}

Let $r \geq 3$, let $\mathsf{C},\mathsf{D} \in \mathrm{NSPD}_{2r}$. Let also $g \in \mathbb{F} \setminus \mathbb{B}_{2r}$ so that $|g| \geq 5$. Assume we have constructed extensions $\mathsf{C}',\mathsf{D}' \in \mathrm{NSPD}_g$ with $\mathfrak{e}(\mathsf{C}',\mathsf{D}') = \mathfrak{e}(\mathsf{C},\mathsf{D})$. \\
\\
Let $\delta > 0$ and choose $\epsilon \in (0, \delta)$ small enough that if $\underline{\mathsf{C}}$ and $\underline{\mathsf{D}}$ are elements of $\mathrm{NSPD}_g$ satisfying \[ \max \bigl( ||\mathsf{C}' - \underline{\mathsf{C}}||_1,||\mathsf{D}'-\underline{\mathsf{D}}||_1 \bigr) \leq \epsilon \] then we have \begin{equation} \label{eq.y} \mathfrak{e}(\underline{\mathsf{C}},\underline{\mathsf{D}}) \leq \mathfrak{e}(\mathsf{C}',\mathsf{D}') + \delta = \mathfrak{e}(\mathsf{C},\mathsf{D})+\delta \end{equation} We can use Lemma \ref{lem.singular} to choose a pair $(\underline{\mathsf{C}},\underline{\mathsf{D}})$ with singular degeneracies satisfying (\ref{eq.y}). Using Proposition \ref{state}, we can then find elements $\mathsf{C}''$ and $\mathsf{D}''$ of  $\mathrm{NSPD}_{g_\downarrow}$ extending $\underline{\mathsf{C}}$ and $\underline{\mathsf{D}}$ respectively such that $\mathfrak{e}(\mathsf{C}'',\mathsf{D}'') \leq \mathfrak{e}(\mathsf{C},\mathsf{D}) + \delta$. \\
\\
The previous argument show that for every $n \in \mathbb{N}$ there exists a pair $(\mathsf{C}_n,\mathsf{D}_n)$ of elements of $\mathsf{NSPD}_{g_\downarrow}$ satisfying \[ \max \bigl( ||\mathsf{C} - (\mathsf{C}_n \rest \mathbb{B}_{2r}) ||_1,||\mathsf{D}-(\mathsf{D}_n \rest \mathbb{B}_{2r}) ||_1 \bigr) \leq \frac{1}{n} \] 

and $\mathfrak{e}(\mathsf{C}_n,\mathsf{D}_n) \leq \mathfrak{e}(\mathsf{C},\mathsf{D}) + \frac{1}{n}$. Then the limit $(\mathsf{C}_\downarrow,\mathsf{D}_\downarrow)$ of a pointwise convergent subsequence of these $(\mathsf{C}_n,\mathsf{D}_n)$ will form extensions of $\mathsf{C}$ and $\mathsf{D}$ to $\mathrm{NSPD}_{g_\downarrow}$ satisfying $\mathfrak{e}(\mathsf{C}_\downarrow,\mathsf{D}_\downarrow) = \mathfrak{e}(\mathsf{C},\mathsf{D})$. This completes the recursive extension procedure.

\bibliographystyle{plain}
\bibliography{/Users/peterburton/Documents/Math/Bibliography/pjburtonbibliography.bib}

\texttt{burtonpeterj@icloud.com}\\
\texttt{kate.juschenko@gmail.com}

\end{document}